\documentclass[12pt, oneside]{amsart}   	
\usepackage[margin= 1in]{geometry}                		
\usepackage{graphicx}				

\title[Mirror Relations from Gromov-Witten relations of root stacks] {Intrinsic mirror symmetry and Frobenius structure theorem via Gromov-Witten theory of root stacks}
\author{Samuel Johnston}
\address{Samuel Johnston, Imperial College London, South Kensington Campus, London SW7 2AZ, UK}
\email{samuel.johnston@ic.ac.uk}
\date{}		

\usepackage{amssymb}
\usepackage{mathrsfs}
\usepackage{amsthm}
\usepackage{amsmath}
\usepackage{tikz-cd}
\usepackage{hyperref}
\usepackage{dsfont}
\usepackage{tikz-3dplot}
\newcounter{thm}
\newtheorem{remark}[thm]{Remark}
\newtheorem{proposition}[thm]{Proposition}
\newtheorem{definition}[thm]{Definition}
\newtheorem{theorem}[thm]{Theorem}
\newtheorem{example}[thm]{Example}

\newtheorem{lemma}[thm]{Lemma}
\newtheorem{corollary}[thm]{Corollary}
\newtheorem{assumption}[thm]{Assumption}
\newcommand{\ZZ} {\mathbb{Z}}
\newcommand{\NN}{\mathbb{N}}
\newcommand{\kk}{\mathds{k}}

\usetikzlibrary{fadings}

\numberwithin{equation}{section}
\numberwithin{thm}{section}
\numberwithin{figure}{section}

\begin{document}

\begin{abstract}
Using recent results of Battistella, Nabijou, Ranganathan and the author, we compare candidate mirror algebras associated with certain log Calabi-Yau pairs constructed by Gross-Siebert using log Gromov-Witten theory and Tseng-You using orbifold Gromov-Witten theory of root stacks. Although the structure constants used to defined these mirror algebras do not typically agree, we show that any given structure constant involved in the construction the algebra of Gross and Siebert can be computed in terms of structure constants of the algebra of Tseng and You after a sequence of log blowups. Using this relation, we provide another proof of associativity of the log mirror algebra, and a proof of the weak Frobenius Structure Theorem in full generality. Along the way, we introduce a class of twisted punctured Gromov-Witten invariants of generalized root stacks induced by log \'etale modifications, and use this to study the behavior of log Gromov-Witten invariants under ramified base change.
\end{abstract}
\maketitle

\section{Introduction}

The past few years have seen much progress in constructing mirrors to log Calabi-Yau varieties directly using algebro-geometric enumerative geometry. Work of Gross and Siebert in \cite{int_mirror} has shown that a particular class of log Gromov-Witten invariants satisfy relations reminicent of the classical WDVV relations in the theory of stable maps, which in turn encode an algebra associated with the log Calabi Yau pair ($X,D$). However, the arguments used to prove the required relations in log Gromov-Witten theory are subtler than one might naively expect, as the required relations do not follow from a straightforward application of WDVV as in the setting of standard Gromov-Witten theory. Seperately, Tseng and You in \cite{RQC_no_log} defined a similar ring using ``large r" orbifold Gromov-Witten invariants. As the orbifold theory does have access to most of the standard techniques from Gromov-Witten theory, the proofs that their algebras satisfy the necessary relations is more straightforward. 

As both the logarithmic and the orbifold Gromov-Witten theories are versions of Gromov-Witten theory relative to a simple normal crossings divisor, it is natural to ask what the relationship is between the theories, and in our specific context understand the relationship between the two mirror algebras constructed using the different relative curve counting theories. Recent work of Battistella, Nabijou and Ranganathan in \cite{logroot} and \cite{logroot2} addresses the first of these questions in genus $0$. Roughly speaking, they identify the difference between the theories to be a result of the fact that log Gromov-Witten theory is invariant under log \'etale modifications, by results in \cite{bir_GW}, a relation which does not hold in the large $r$ orbifold theory. Moreover, for any given log invariant of interest associated with a log smooth target $X$, there exists a log \'etale modification $\pi \colon \tilde{X} \rightarrow X$ such that the log invariant on $X$ is equal to an orbifold invariant on $\tilde{X}$. Separately, work of the author in \cite{pbirinv} has extended the main result of \cite{bir_GW} to the negative contact setting, and explicitly described how the log Gromov-Witten theory of a target changes under log \'etale modifications, both at the level of virtual cycles and moduli stacks. This result in particular demonstrated a birational invariance result for the intrinsic mirror algebra constructed in \cite{int_mirror}. 

Using these results, we initiate a study into the second of the questions above, in the process observing relations in log Gromov-Witten theory analogous to standard relations in ordinary Gromov-Witten theory.  More precisely, after first providing simple examples demonstrating the structure constants defining these algebras differ in general, the following theorem makes clear how these classes of invariants are related nonetheless:

\begin{theorem}\label{mthm1}
For $(X,D)$ a log Calabi-Yau variety in the sense of \cite{int_mirror}, with log and orbifold mirror algebras $R^{log}_{(X,D)}$ and $R^{orb}_{(X,D)}$ respectively, then any structure constants $N_{p,q,r}^\textbf{A}$ defining the product on $R^{log}_{(X,D)}$ in the theta basis is equal to a sum of orbifold structure constants on a log \'etale modification. 
\end{theorem}

Using this fact, we give a new proof of the associativity of the mirror algebra introduced by Gross and Siebert in \cite{int_mirror}, and a proof of the weak Frobenius structure theorem in full generality i.e. Conjecture $9.2$ of \cite{int_mirror}, which relates the natural trace form on $R_{(X,D)}$ with certain decendent log Gromov-Witten invariants. We recall both theorems below:

\begin{corollary}\label{mcr1}
Given a log Calabi-Yau pair $(X,D)$, then the structure constants $N_{p,q,r}^\textbf{A}$ defined in \cite{int_mirror} satisfy the following relation:

\[ \sum_{\textbf{A} = \textbf{A}_1 + \textbf{A}_2\text{, }s \in B(\mathbb{Z})} N_{p_1,s,r}^{\textbf{A}_1}N_{p_2,p_3,s}^{\textbf{A}_2} = \sum_{\textbf{A} = \textbf{A}_1 + \textbf{A}_2\text{, }s \in B(\mathbb{Z})} N_{p_1,p_2,s}^{\textbf{A}_1}N_{p_3,s,r}^{\textbf{A}_2}.\]
Hence, the structure constants define an associative product on $R_{(X,D)}$. 

\end{corollary}

%

\begin{theorem}[Conjecture $9.2$ \cite{int_mirror}]\label{mthm2}
Let $\beta$ be a type of tropical map with curve class $\textbf{A}$ and $k+1$ marked points $x_{out},x_1,\ldots,x_k$ with contact order $0,p_1,\ldots,p_k$ respectively associated with integral points of $B$. Define the invariant:
\[N_{p_1,\ldots,p_k}^{\textbf{A}} = \int_{[\mathscr{M}(X,\beta)]^{vir}} \psi_{x_{out}}^{k-2}ev_{x_{out}}^*(pt)\]
 Then the $\vartheta_0$ coefficient of $\vartheta_{p_1}\cdots\vartheta_{p_k}$ is:
\[\sum_{\textbf{A}} N_{p_1,\ldots,p_k}^\textbf{A}z^\textbf{A}\]

\end{theorem}

We emphasize that key to the proofs of the results above are exploiting the distinct natures of the combinatorial data which govern the log and orbifold Gromov-Witten theories. Invariants in log Gromov-Witten theory frequently decompose into contributions coming from fixed tropical data, but due to subtleties which include gluing of log maps, we typically cannot pullback relations between invariants coming from moduli of stable maps. Orbifold Gromov-Witten theory does possess relations derived from moduli of stable curves, but typically lacks the full tropical refinement present in the logarithmic theories. The relative disadvantages of either theory of circumvented by exploiting the log/orbifold comparison. The full proof for the results above must utilize the distinct advantages of each theory; it is not simply a matter of translating relations from orbifold Gromov-Witten theory into log Gromov-Witten theory.

%

To begin with, we will recall the relative Gromov-Witten invariants used to define the orbifold and log mirror algebras, and highlight through example how these classes of invariants typically do not coincide. We then introduce twisted log Gromov-Witten theory for root stacks defined by certain log \'etale modifications and generalize the main result of \cite{pbirinv} which will be useful in this setting. Following this, we proceed to a proof of Theorem \ref{mthm1}, and prove Corollary \ref{mcr1} assuming a semipositivity condition used by Tseng and You in \cite{RQC_no_log}. To remove this assumption, we provide an argument for the vanishing of certain orbifold invariants based on the vanishing of certain log invariants, modifying an argument which appeared in \cite{int_mirror}. This vanishing result in turn allows us to use the WDVV and TRR relations in the orbifold theory to conclude the desired relations in the log theory.

\subsection{Future Work}

Along with constructing a deformation of the Stanley-Reisner ring of $\Sigma(X)$ from large $r$ orbifold invariants of a log Calabi-Yau pair $(X,D)$, Tseng and You also define a quantum product on relative quantum cohomology for arbitrary simple normal crossings pairs $(X,D)$. This is a graded algebra structure on the cohomology of the inertia stack for arbitrarily large $r$. In the log Calabi-Yau setting, this product rule restricts to define a product on the degree $0$ part of the ring. The relations in this paper suggest an analogous construction in log Gromov-Witten theory, and a more thorough treatment of the state space should enable this to be accomplished. 


Additionally, You in \cite{relq_bir} has previously studied in various settings how the relative orbifold quantum cohomology ring changes under birational modifications. Motivated by the crepant transformation conjecture, he predicts that while individual invariants might not satisfy any reasonable interpretation of birationally invariance, their generating series might be better behaved, with their difference controlled by analytic continuation in Novikov parameters.  A more thorough understanding of the behavior of orbifold invariants under strata blowups, when combined with the results in this paper, should yield additional tools and techniques for computing and studying the intrinsic mirror construction.

The relations in this paper should also hold for more general log Calabi-Yau DM stacks. The main missing ingredient needed to carrying this out presently is the development of a well-behaved theory of twisted log maps to log DM stacks. Once this is accomplished though, the arguments in this paper should generalize to the DM stack setting, and in particular produce mirror algebras for these targets. A generalized mirror construction will likely be useful in studying mirrors to skew-symmetrizable cluster varieties in the same vein as \cite{clustermirr}, as well as mirror LG models to smooth Fano orbifolds. We will return to these questions in future work.

\section{Acknowledgements}
I would like to thank Robert Crumplin, Mark Gross, Ajith Kumaran, Navid Nabijou, Dhruv Ranganathan and Fenglong You for many useful discussions related to this work. In particular, I would like to thank Mark Gross for reviewing an earlier draft of this document, and Navid Nabijou and Dhruv Ranganathan for keeping me updated on their progress with their joint work with Luca Battistella. This work was done with the financial support of the ERC grant ``Mirror Symmetry in Algebraic Geometry" as well as the Heilbronn Institute for Mathematical Research and Imperial College London.

\section{Tropical and Logarithmic preliminaries}
We suppose throughout that we are working over an algebraically closed field $\kk$ with characteristic $0$. We begin by briefly reviewing the necessary tropical and logarithmic geometry we require, as well as review the main results of \cite{logroot} and \cite{logroot2}. Given an algebraic stack with log structure $Y$, we denote the underlying stack over schemes by $\underline{Y}$ throughout.

\subsection{Tropicalization and Artin Fans}
Consider a pair $(X,D)$ with $X$ a smooth projective variety and $D$ a simple normal crossings divisor with $D = D_1 + \ldots + D_k$ a decomposition of the divisor $D$ into irreducible components. Each $D_i$ is  determined by a section $s_i$ of a line bundle $L_i \in Pic(X)$ up to scaling by $\kk^*$. By definition, this induces a morphism $(L_i,s_i): X \rightarrow [\mathbb{A}^1/\mathbb{G}_m]$, and taking the product of these morphisms induces a morphism $X \rightarrow [\mathbb{A}^k/\mathbb{G}_m^k]$. We let $\mathcal{X}$ be the minimal open substack of $[\mathbb{A}^k/\mathbb{G}_m^k]$ through which the above morphism factors. Each of the stacks above admit natural divisorial log structures and the log algebraic stack $\mathcal{X}$ is called the Artin fan for the pair $(X,D)$. 

The use of Artin Fans in logarithmic geometry allows us to translate constructions and calculations in a more combinatorial/tropical setting into the geometric setting. In this paper, we consider the tropical setting to be the category of rational polyhedral cone complexes, or \textbf{RPCC}. Its objects consist of rational polyhedral cone complexes, which are given by collections of integral convex cones glued along face inclusions. For a cone $\sigma$, we denote its integral points by $\sigma_\NN$. For a cone complex $\Sigma$, we denote by $\Sigma(\ZZ)$ the set of morphisms $\ZZ \rightarrow \sigma_{\NN}^{gp}$ for some $\sigma \in \Sigma$. By working in the more general context of cone stacks, defined in \cite{stacktrop} Definition $2.7$, a useful result of \cite{stacktrop} is the following equivalence of categories:

\begin{lemma}\label{equiv}
There is an equivalence of $2$-categories between the categories of Artin fans and cone stacks. On a cone $\sigma \in Ob(\textbf{RPCC})$, the equivalence is given by sending $\sigma$ to $\mathcal{A}_\sigma$. 
\end{lemma}

Given a cone complex $\Sigma$, we will refer to its corresponding Artin fan as $\mathcal{A}_{\Sigma}$. We will also occasionally refer to idealized Artin cones in this paper, denoted by $\mathcal{A}_{\sigma,K}$ for $K \subset \sigma^{\vee}_{\NN}$ an ideal. This is an idealized log stack whose underlying stack is given by:

\[\mathcal{A}_{\sigma,K} = [(Spec\kk[\sigma^{\vee}_\NN]/K)/Spec\kk[\sigma_\NN^{gp}]],\]
and which carry canonical log and idealized log structures. 

When $(X,D)$ is a simple normal crossings pair, forgetting the morphism $X \rightarrow \mathcal{X}$ and remembering only the target is the same information as the dual cone complex of $(X,D)$, which we denote by $\Sigma(X)$. The forgetful morphism $\mathcal{X} \rightarrow [\mathbb{A}^1/\mathbb{G}_m]$ remembering only a single divisor $D_i$ corresponds under the correspondence to a morphism of rational polyhedral cone complexes $\Sigma(X) \rightarrow \mathbb{R}_{\ge 0}$ i.e. a PL function on the complex $\Sigma(X)$. We call the function induced by a divisor $D$ by $D^*$. If we restrict to a cone $\sigma \in \Sigma(X)$, this morphism induces a group homomorphism on integral tangent vectors at any point $p \in int(\sigma)$ contained in the interior of $\sigma$, which we may identify with a group homomorphism $\sigma_\NN^{gp} \rightarrow \ZZ$.  

One important input coming from this correspondence above is the relation between subdivisions and refinements of the Artin fan and log \'etale modifications of $X$. Given a cone $\sigma$ and a conic subdivision $\tilde{\sigma}\rightarrow \sigma$, we have a corresponding morphism of Artin fans $\mathcal{A}_{\tilde{\sigma}} \rightarrow \mathcal{A}_{\sigma}$. By \cite{bir_GW} Corollary $2.6.7$, this morphism is log \'etale and proper. Moreover, if we have a strict smooth map $X \rightarrow \mathcal{A}_\sigma$, we may pullback this map to produce a proper log \'etale and birational map $\tilde{X} \rightarrow X$. Alternatively, if we consider a morphism of cones $\sigma' \rightarrow \sigma$ induced by a refinement of lattices, we may pullback as before to produce a log \'etale modification. In the setting of a pair $(X,D)$ above, these produce orbifold enhancements of the target which we call generalized root stack constructions along the boundary divisor $D$. 
More generally, pulling back subdivisions and refinements of cone complex $\mathcal{X}$ induce log \'etale modifications of $X$. By the birational invariance of punctured log invariance, we are free to choose a convenient birational model of $X$ of the form above to study relations in the log Gromov-Witten invariants of $(X,D)$ of interest.

\subsection{Tropical maps and punctured log maps}

Given a log Calabi-Yau pair $(X,D)$, the enumerative geometry used to construct its mirror algebra is encoded in the theory of punctured log maps, constructed by Abramovich, Chen, Gross and Siebert in \cite{punc}, which we briefly recall below, omiting details which will not be important for us in this context:

\begin{definition}\label{defpunc}
Let $Y = (\underline{Y},\mathcal{M}_Y)$ be an fs log stack, together with a decomposition $\mathcal{M}_Y = \mathcal{M} \oplus_{\mathcal{O}_Y^\times} \mathcal{P}$. A puncturing of $Y$ along $\mathcal{P}$ is a fine subsheaf of monoids $\mathcal{M}_{Y^\circ} \subset \mathcal{M} \oplus_{\mathcal{O}_Y^\times} \mathcal{P}^{gp}$, containing $\mathcal{M}_Y$, together with a structure map $\alpha_{Y^\circ}:\mathcal{M}_{Y^\circ} \rightarrow \mathcal{O}_Y$, satisfying:

\begin{enumerate}
\item The inclusion $p^\flat: \mathcal{M}_Y \rightarrow \mathcal{M}_{Y^\circ}$ is a morphism of log structures on $\underline{Y}$.

\item For any geometric point $x \in Y$, let $s_x \in \mathcal{M}_{Y^\circ,x}\setminus \mathcal{M}_{Y,x}$. Representing $s_x = (m_x,p_x) \in \mathcal{M}_x \oplus_{\mathcal{O}^\times_{Y_x}} \mathcal{P}^{gp}_x$, we have $\alpha_{Y^\circ}(s_x) = \alpha_{\mathcal{M}}(m_x) = 0$ in $\mathcal{O}_{Y,x}$.

\end{enumerate}

A punctured log curve is a log curve $C \rightarrow S$, together with a puncturing along $\mathcal{P}$ induced by marked sections of $C \rightarrow S$. A stable punctured log map $C^\circ \rightarrow X$ is a logarithmic morphism from a prestable punctured log curve to $X$ such that the underlying morphism of schemes is a stable map.
\end{definition}

Similar to the theory of ordinary log stable maps, every punctured log map has a ``combinatorial shadow" induced by tropicalization. More precisely, to any punctured log stable map $f: C \rightarrow X$ over a logarithmic point $S$ and letting $\sigma = Hom_{mon}(\overline{\mathcal{M}}_S,\mathbb{R}_{\ge 0})$, the log structure encodes a metrization of the dual graph of $C$ by a length function $l: E(G) \rightarrow \overline{\mathcal{M}}_S$, and a PL function $\Sigma(f)_p: G \rightarrow \Sigma(X)$ for every point $p \in \sigma$. We note such data gives the following discrete invariants:

\begin{enumerate}
\item A function $\pmb\sigma(-):E(G)\cup L(G)\cup V(G) \rightarrow \Sigma(X)$ assigning to a feature of $G$ the smallest cone of $\Sigma(X)$ which contains it's image under $\Sigma(f)$.
\item For every edge or leg $e \in E(G) \cup L(G)$ an integral tangent vector $u_e \in \pmb\sigma(e)^{gp}_\NN$. 
\end{enumerate}
We may also keep track of the genus distribution on the dual graph by a function $g: V(G) \rightarrow \NN$. In this paper, we will always be working with genus $0$ domains, so this additional data will be forgotten from now on. The data $(G,l,\pmb\sigma,u)$ defines the tropical type of the log stable map $f$. A choice of the tuple of data above is called realizable if there is a PL map from meterized dual graph $G$ to $\Sigma(X)$ of the given tropical type. In such a situation, there exists a cone $\tau$ and family of family of tropical maps $(f,\Gamma)$ defined over $\tau$ which are the moduli spaces of tropical curves of type $\tau$ and the universal family lying over it respectively. The primary utility that punctured log maps provide tropically is that the legs are possibly bounded. 

Applying functorial tropicalization to a punctured log map $C^\circ \rightarrow X$ produces a map of cone complexes $\Sigma(C) \rightarrow \Sigma(X)$, which is a punctured tropical map defined as follows:

\begin{definition}
Given a family of ordinary tropical curves $\pi_{(G,l)}: \Gamma_\sigma \rightarrow \sigma$, a puncturing $\Gamma_\sigma^\circ \rightarrow \sigma$ of $\Gamma_\sigma$ is an inclusion of a sub cone complex $\Gamma_{\sigma}^\circ \rightarrow \Gamma_\sigma$ over $\sigma$ which is an isomorphism away from cones of $\Gamma_\sigma$ associated with legs of $G$. Moreover, for cones $\sigma_l$ associated with punctured legs of $G$, $\sigma_l$ must intersect the interior of the target cone, and there exists a section $\sigma \rightarrow \sigma_l$ which when composed with $\sigma_l \rightarrow \Gamma_\sigma$ is a section associated with the unique vertex of $G$ contained in the leg $l$. 
\end{definition}


We bundle the data of the contact orders associated to legs $u_l$ as well as a degree $\textbf{A} \in H_2(X)$ into a pair $\beta = ((u_l),\textbf{A})$ which we call a type of punctured log map. The main result of \cite{punc} is a construction of a finite type DM stack $\mathscr{M}(X,\beta)$ with log structure parameterizing punctured log stable maps to $X$ of type $\beta$. There is also an algebraic stack $\mathfrak{M}(\mathcal{X},\beta)$ with log structure parameterizing prestable punctured log maps of type $\beta$ to the Artin fan $\mathcal{X}$, and composition induces a map $\mathscr{M}(X,\beta) \rightarrow \mathfrak{M}(\mathcal{X},\beta)$. This morphism carries a perfect obstruction theory, which facilitates virtual pullback of classes on $\mathfrak{M}(\mathcal{X},\beta)$ in the formalism of \cite{vpull}. 

It will be useful for various arguments in this paper to enrich the target of morphism carrying the perfect obstruction theory. Specifically, given a punctured log curve $C \rightarrow \mathcal{X}$, recall that a marked point $p$ of a puncture with contact order $u_p$ is forced to map to a stratum $\mathcal{D}_p$ of the log stack $\mathcal{X}$.  Thus, on the level of underlying stacks, we have a morphism $ev_p: \underline{\mathfrak{M}(\mathcal{X},\beta)} \rightarrow \underline{\mathcal{D}_p}$. Moreover, the stratum of the Artin fan of $X$ specifies a stratum $D_p$ of $X$, and the strict morphism $X \rightarrow \mathcal{X}$ restricts to a morphism $D_p \rightarrow \mathcal{D}_p$. Hence, for a collection of marked points $p_1,\ldots p_k$, of the type $\beta$, we define:

\[\mathfrak{M}^{ev}(\mathcal{X},\beta) := \mathfrak{M}^{ev(p_1,\ldots,p_k)}(\mathcal{X},\beta) = \mathfrak{M}(\mathcal{X},\beta)\times_{\prod{i=1}^k\underline{\mathcal{D}_{p_i}}} \prod_{i=1}^k \underline{D_{p_i}}\]
The morphism $\mathscr{M}(X,\beta) \rightarrow \mathfrak{M}(\mathcal{X},\beta)$ certainly factors through the projection $\mathfrak{M}^{ev}(\mathcal{X},\beta)$, and the resulting morphism $\mathscr{M}(X,\beta)\rightarrow \mathfrak{M}^{ev}(\mathcal{X},\beta)$ also carries a perfect obstruction theory. 

\subsection{Construction of virtual classes on $\mathscr{M}(X,\beta)$}

An initially troubling feature of theory of punctured log maps is the fact that if we fix only a type $\beta$, unlike in the theory of ordinary log maps, we do not have a canonical virtual fundamental class on $\mathscr{M}(X,\beta)$. Indeed, the moduli stack $\mathfrak{M}(\mathcal{X},\beta)$ might not be equidimensional. There are at least two paths for dealing with this issue. One is to note that by Proposition $3.28$ of \cite{punc}, if we fix additional tropical data in the form of a realizable tropical type $\tau$, then the moduli stack $\mathfrak{M}(\mathcal{X},\tau)$ of punctured log maps of type $\tau$ is equidimensional. Moreover, the moduli stack $\mathfrak{M}(\mathcal{X},\beta)$ is stratified by realizable tropical types of type $\beta$, so we produce classes on $\mathscr{M}(X,\beta)$ associated with every such tropical type. Note that the dimension of these classes will differ, and can be computed using the dimension of the associated tropical moduli problem

An alternative approach to equipping this moduli stack with an equidimensional class is suggested by Battistella, Nabijou and Ranganathan in \cite{logroot2}. This assumes the target is \emph{slope sensitive}, a property originally introduced in \cite{logroot} Section $4.1$ and $4.2$. Critically, for any discrete data $\beta$, this property can be assured for a corresponding set of discrete data $\beta'$ on a log \'etale modification $\tilde{X} \rightarrow X$.  Assuming this condition, we may also equip $\mathfrak{M}(\mathcal{X},\beta)$ with a virtual fundamental class. By taking the virtual pullback of this class, we produce a class on $\mathscr{M}(X,\beta)$. Letting $D_i$ be an irreducible component of $D$, and $e_i = |\{p\text{ marked point}| D_i^*(u_p))<0\}|$, the dimension of this class is given by:

\begin{equation}\label{vdimorb}
vdim\text{ }\mathscr{M}(X,\beta) = deg\text{  }c_1(T_X^{log}) + (dim\text{ }X-3)(1-g) + n - \sum_{i=1}^k e_i.
\end{equation}

%
%

\section{Intrinsic Mirror constructions}
Associated to the pair $(X,D)$ is the rational polyhedral cone complex $\Sigma(X)$. When $X$ is smooth and $D$ is simple normal crossings, this is the dual cone complex of the divisor $D$. We recall from \cite{int_mirror} two possible restrictions on the geometry of pair $(X,D)$ that are sufficient to construct mirror algebras:

\begin{assumption}\label{lcy}
An admissible pair is a log smooth pair $(X,D)$ which satisfies one of two of the following conditions:
\begin{enumerate}
\item $c_1(T_X^{log})$ is nef or antinef.
\item $K_X + D =_\mathbb{Q} \sum_i a_iD_i$, with $a_i \ge 0$.
\end{enumerate}
If $(X,D)$ satisfies the second of the above conditions, we say $(X,D)$ is log Calabi-Yau. 
\end{assumption}

If the pair $(X,D)$ is log Calabi-Yau in the second sense, we call a divisor $D_i$ \emph{good} if $a_i = 0$. We adopt the terminology that divisors which are not good are called \emph{bad}. We define $(B,\mathscr{P})$ to be the subcomplex of $\Sigma(X)$ consisting of the cones of $\Sigma(X)$ associated to strata which are contained only in good divisors. If $(X,D)$ satisfies condition $1$ above, then we define $(B,\mathscr{P})$ to be the cone complex $\Sigma(X)$. 

To define the logarithmic mirror algebra, we first pick a finitely generated monoid $Q \subset H_2(X)$ containing all effective curve classes, and $I \subset \kk[Q]$ a monomial ideal such that $\sqrt{I}$ is the maximal monomial ideal of $\kk[Q]$. We write $\kk[Q]/I = S_{X,I}$. The mirror algebra $R^{log}_{(X,D)}$ first will be a free $S_{X,I}$ module:

\[R^{log}_{(X,D)} = \oplus_{p \in B(\ZZ)} \vartheta_p\cdot S_{X,I}.\]

The generators of the free $S_{X,I}$ module above are called theta functions. We equip $R^{log}_{(X,D)}$ with an $S_{X,I}$ algebra structure via structure constants for the product of theta functions. For $p,q,s \in B(\ZZ)$, the structure constant $N_{p,q,s} := \sum_{\textbf{A}} N_{p,q,s}^\textbf{A} z^\textbf{A} \in S_{X,I}$ given by the coefficient appearing in front of $\vartheta_s$ in the product $\vartheta_p\vartheta_q$ will be a certain series of $3$-pointed punctured Gromov-Witten invariants indexed by curve classes, with two positive contact orders $p$ and $q$, and one negative contact order $-s$.

 More precisely, we first consider the evaluation space for log points with contact order $-s$. This is a logarithmic stack, whose underlying stack is given by $X_s \times B\mathbb{G}_m$, with $X_s$ the closure of the locally closed stratum $X_s^\circ \subset X$ associated with the smallest cone $\sigma \in \Sigma(X)$ containing $s$, and whose log structure is a sub log structure of the pullback of the product log structure $X\times [\mathbb{A}^1/\mathbb{G}_m]$, whose ghost sheaf is given by:

\begin{equation}\label{logev}
\Gamma(U,\overline{\mathcal{M}}_{\mathscr{P}(X,s)}) = \{(m,s(m)) \text{ }|\text{ }m \in \Gamma(U,\overline{\mathcal{M}}_Z)\} \subset \Gamma(U,\overline{\mathcal{M}}_Z) \oplus \mathbb{N} = \Gamma(U,\overline{\mathcal{M}}_{\widetilde{\mathscr{P}(X,s)}})
\end{equation}
In the terminology of \cite{punc}, for every curve class $\textbf{A}$ associated with a non-zero element of $S_{X,I}$, there is a non-realizable decorated tropical type $\beta$ with a single vertex decorated by $\textbf{A}$, and $3$ legs with contact orders $p_1$, $p_2$, and $-s$. There is a corresponding moduli space of punctured log maps $\mathscr{M}(X,\beta)$, which by \cite{int_mirror} Proposition $3.3$ admits an evaluation map $\mathscr{M}(X,\beta) \rightarrow \mathscr{P}(X,s)$. We let $W = B\mathbb{G}_m^\dagger$ be the log stack with underlying stack $B\mathbb{G}_m$ and log structure induced by pullback along the closed immersion $B\mathbb{G}_m \rightarrow [\mathbb{A}^1/\mathbb{G}_m]$. Letting $z \in X_s$, then by Proposition $3.8$ of \cite{int_mirror}, there is a map $W \rightarrow \mathscr{P}(X,s)$ which maps $W$ to $z \times W \subset \mathscr{P}(X,s)$. We then consider the fine and saturated fiber product $\mathscr{M}(X,\beta,z):= \mathscr{M}(X,\beta) \times_{\mathscr{P}(X,s)}^{fs} W$. We may similarly construct $\mathfrak{M}^{ev}(\mathcal{X},\beta,z) := \mathfrak{M}^{ev}(\mathcal{X},\beta)\times_{\mathscr{P}(X,s)}^{fs} W$, and the morphism $\mathscr{M}(X,\beta,z) \rightarrow \mathfrak{M}^{ev}(\mathcal{X},\beta,z)$ pulled back from $\mathscr{M}(X,\beta)\rightarrow \mathfrak{M}^{ev}(\mathcal{X},\beta)$ carries the pulled back obstruction theory. In particular, when $(X,D)$ is log Calabi-Yau, $\mathscr{M}(X,\beta,z)$ comes with a virtual fundamental class of virtual dimension $0$. The structure constant $N_{p,q,r}^\textbf{A}$ are then defined as: 
\[N_{p_1,p_2,s}^\textbf{A} = deg [\mathscr{M}(X,\beta,z)]^{vir}.\] 
If $(X,D)$ satisfies condition $1$ of Assumption \ref{lcy}, then the structure constant is as above if $vdim\text{ }\mathscr{M}(X,\beta,z) = 0$, and is defined to be zero otherwise. The main result of \cite{int_mirror} is that these structure constants define an associative product on $R^{log}_{(X,D)}$.

In the orbifold setting, assuming in addition that the boundary divisor $D$ has simple normal crossings and is an admissible pair in the sense of Definition \ref{lcy}(1), Tseng and You in \cite{RQC_no_log} define analogous invariants giving these structure constants. In loc cit., the authors observe that the contact orders $p,q,-s$ and curve class $\textbf{A}$ may also be used to construct a sequence of moduli stack $\mathscr{M}^{orb}(X_r,\beta)$, with $r = (r_i)$ a vector of pairwise relatively prime positive integers indexed by the connected components of the boundary and $X_r$ the corresponding stack theoretic root of the boundary divisor associated with $r$. In the orbifold setting, the negative contact orders are interpreted as a high age orbifold point. The moduli space $\mathscr{M}^{orb}(X_r,\beta)$ is a Deligne-Mumford stack with a perfect obstruction theory with $vdim\text{ }\mathscr{M}^{orb}(X_r,\beta) = dim\text{ }D_s$. 

To fix notation which will be used throughout, we let $H^*(-)$ be the singular cohomology functor, and pick a basis $\{T_{i,s}\} \in H^*(D_s)$ for the singular cohomology of the stratum associated with $s \in \ZZ^k$ a contact order with $D = \sum_{i=1}^k D_i$. Since we are assuming $X$ is simple normal crossings, there is a Poincar\'e dual basis $T_s^i$. Moreover, given a collection of contact orders $p_i \in \ZZ^k$, $\alpha_{i} \in H^*(D_{p_i})$, $\textbf{A} \in NE(X)$, and setting $m_{i,-}$ be the number of negative entries appearing in the contact order $m_i$, we define:

\begin{equation}\label{orbinv}
\langle \prod_i \psi^{a_i}[\alpha_i]_{p_i}\rangle_{\textbf{A}} := \prod_i r^{m_{i,-}}\int_{[\mathscr{M}^{orb}(X_r,\textbf{A},\textbf{p})]^{vir}} \prod_i\psi_{x_i}^{a_i}\cup ev_{x_i}^*([\alpha_i]).
\end{equation}

By Corollary $18$ of \cite{RQC_no_log}, this invariant stabilizes for $r_i >> 0$. In light of this stabilization, these invariants are usefully thought of as a Gromov-Witten invariant of the infinite root stack associated with the pair $(X,D)$. The authors of loc. cit. define the orbifold structure constants by:

\[N_{p_1,p_2,s}^{\textbf{A},orb}=\langle[1]_{p_1},[1]_{p_2},[pt]_{-s}\rangle_{\textbf{A}}.\]
In Theorem $37$ of \cite{RQC_no_log}, the authors show that these also define an associative product rule on $R^{log}_{(X,D)}$, using arguments involving the WDVV relations in orbifold Gromov-Witten theory. 

As logarithmic and orbifold Gromov-Witten theory are distinct methods at probing the enumerative geometry of the pair $(X,D)$, it is natural to ask about the relationship between these two enumerative theories. In genus $0$, the relationship was described in \cite{logroot} in the setting of non-negative contact orders, and \cite{logroot2} in the setting of general contact orders. We record their main result below:

\begin{theorem}\label{neglogorb}
Let $\beta$ be a tropical type which fixes a curve class $\textbf{A} \in NE(X)$ and tangency orders $p_1,\ldots,p_l$. Then assuming $\beta$ is slope sensitive and letting $[\mathscr{M}(X,\beta)]^{vir}$ be the refined virtual class on the punctured log moduli space, we have:

\[\int_{[\mathscr{M}(X,\beta)]^{vir}} \prod_i \psi_{x_i}^{a_i}ev_{x_i}^*([\alpha_i]) = \langle\prod_i \psi_{x_i}^{a_i}[\alpha_i] \rangle.\]

\end{theorem}

In the context of this paper, we will be interested in the following two corollaries:

\begin{corollary}
Given a type $\beta$ of log stable map to $X$ which is slope sensitive with only one negative contact order, then $\mathscr{M}(X,\beta)$ is virtually equidimensional. Moreover, the log Gromov-Witten and large $r$ orbifold Gromov-Witten invariants with respect to $\beta$ coincide. 
\end{corollary}

\begin{corollary}\label{vanishing}
Given a type $\beta$ of log stable map to $X$ which is slope sensitive and for all realizable tropical types $\tau$ marked by $\beta$ we have $[\mathscr{M}(X,\tau)]^{vir} = 0$, then $[\mathscr{M}(X,\beta)]^{vir} = 0$. 
\end{corollary}

Before proceeding, we record the following relation between curve classes and contact orders, which holds in both the settings of log Gromov-Witten theory and orbifold Gromov-Witten theory, which follows from the weak balancing conditions in both settings as well as Remark $2.11(2)$ of \cite{int_mirror}:

\begin{proposition}\label{balancing}
Suppose we have a type $\beta = (\textbf{A},(s_i))$, hence $\textbf{A} \in NE(X)$ and $s_i \in \sigma_{i,\NN}^{gp}$ for cones $\sigma_i \in \Sigma(X)$. Letting $\mathscr{M}(X,\beta)$ refer to either the corresponding space of log or orbifold stable maps, then $\mathscr{M}(X,\beta) \not= \emptyset$ only if $D_j \cdot \textbf{A} = \sum_i D_j^*(s_i)$ for all irreducible components $D_j \subset D$. In particular, if all but one contact order is fixed except for a final contact order $s_l$, then there are only finitely many possible choices for a contact order $s_l$. 

\end{proposition}

\section{Counterexample to equality of theta basis}


In order to highlight how these two constructions differ, we consider the following simple example: 

Let $(X_t,D_t)$ the the log Calabi-Yau pair of $\mathbb{P}^2$ and the toric boundary, $D_t = D_{t,1} + D_{t,2} + D_{t,3}$ and $X = Bl_{x} \mathbb{P}^2$, with $x \in D_{t,1}$ a point which is not a zero strata of $(X_t,D_t)$. Finally, let $D$ be the strict transform of $D_t$. Then $(X,D)$ is a log Calabi-Yau pair satisfying $K_X + D = 0$ and tropicalization $\Sigma(X)$ which can be identified with $\Sigma(X_t)$. Hence, both $R_{(X,D)}$ and $R_{(X,D)}^{orb}$ exist, and are algebras over $S_X = \kk[NE(X)]$ which are free as $S_X$ modules with basis integral points of $\Sigma(X)$. Moreover, $D$ is an ample divisor, hence $X\setminus D$ is affine. We present two phenomenon which obstruct the equality of structure constants without passing to appropriate blowups:


\textbf{Case 1: Non-equality at the level of moduli stacks}

Let $L,E \in NE(X)$ be the class of a generic line and the exceptional curve class respectively, $\textbf{A} = 3L - 2E$, $p_1 = D_3 + D_1$ and $p_2 = 3D_2 + 2D_3$. We wish to compute the $\vartheta_0$ term of $\vartheta_{p_1}\vartheta_{p_2}$ of class $\textbf{A} \in NE(X)$, i.e. $N_{p_1,p_2,0}^{\textbf{A}}$ and $N_{p_1,p_2,0}^{\textbf{A},orb}$.

To compute this invariant in the log setting, we recall from \cite{GHS} Theorem $3.24$ that $\vartheta_{p_1}\vartheta_{p_2}[\vartheta_0]$ can be computed by summing over all pairs of broken lines in the canonical wall structure associated with $(X,D)$ with initial slopes given by $p_1$ and $p_2$ which meet at a fixed generic point of $\Sigma(X)$ with final slopes which sum to zero. In this setting, the number of broken lines with fixed initial slope which end at a fixed generic point is finite. In particular, if we pick our generic point in the cone of $\Sigma(X)$ spanned by $D_1$ and $D_3$, then there is a unique broken line associated with $p_1$, and two broken lines associated with $p_2$, $\beta_1$ and $\beta_2$, see the diagram below for an illustration of these broken lines, together with labeling by the slope of their final linear segments. However, neither of the two final slopes of $\beta_i$ is $-p_1$. Since $p_1$ is the slope of the unique linear segment associated with the unique broken line with initial slope $p_1$, it follows that $N_{p_1,p_2,0}^{\textbf{A}} = 0$.

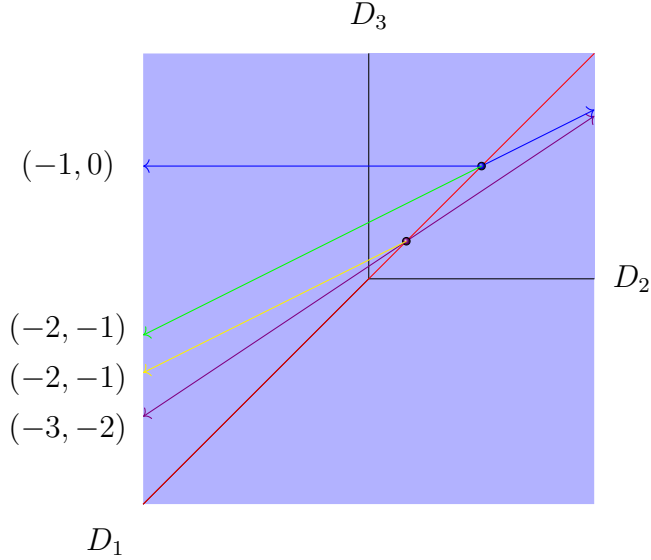
\begin{figure}[h]
\centering
\begin{tikzpicture}
\fill[white!70!blue, path fading = north ] (-6,0)--(-3,0)--(-3,3)--(-6,3)--cycle;
\fill[white!70!blue, path fading = south ] (-6,0)--(-3,0)--(-3,-3)--(-9,-3)--cycle;
\fill[white!70!blue, path fading = north] (-6,0)--(-9,-3)--(-9,3)--(-6,3)--cycle;
\draw[black] (-6,0)--(-6,3);
\node at (-6,3.5) {$D_3$};
\draw[black] (-6,0)--(-3,0);
\node at (-2.5,0) {$D_2$};
\draw[black] (-6,0)--(-9,-3);
\node at (-9.5,-3.5) {$D_1$};
\draw[red] (-3,3)--(-9,-3);
\draw[->,color = violet] (-5.5,.5)--(-3,13/6);
\draw[ball color = violet] (-5.5,.5) circle (0.5mm);
\draw[->, color = violet] (-5.5,.5)--(-9,-11/6);
\node at (-10,-4/3) {$(-2,-1)$};
\draw[->,color = yellow] (-5.5,.5)--(-9,-5/4);
\node at (-10,-2) {$(-3,-2)$};
\draw[->,color = blue] (-4.5,1.5)--(-3,2.25);
\draw[ball color = blue] (-4.5,1.5) circle (0.5mm);
\draw[->,color = blue] (-4.5,1.5)--(-9,1.5);
\node at (-10,1.5) {$(-1,0)$};
\draw[->,color = green] (-4.5,1.5)--(-9,-.75);
\node at (-10,-2/3) {$(-2,-1)$};
\end{tikzpicture}
\caption{The broken lines contributing to $\vartheta_{p_2}$ and $\vartheta_{p}$, which all have a unique bending vertex, colored violet if it contributes to $\vartheta_{p_2}$ and colored blue if it contributes to $\vartheta_{p}$. We color the unique ray in the scattering diagram red, and final rays of the broken lines by blue, violet, yellow and green in order to differentiate.} 
\end{figure}

Before moving on to computing the corresponding orbifold invariant, we also note the following broken line expansion for $\vartheta_p$, with $p = 2D_2 + D_3$. 

\[\vartheta_p = x^{-D_1 - D_3}z^{2L - E} + x^{-2D_1 - D_3}z^{2L}\]

In particular, $\vartheta_{p_1}\vartheta_p[\vartheta_0] = z^{2L - E}$. By the main theorem of \cite{mirrcomp}, this implies there is a unique log curve with $3$ marked points of contact orders $p_1,p$ and $0$ of curve class $2L-E$ which maps the contact order $0$ point to a general point of $X\setminus D$.

To investigate the corresponding structure constant of $R_{(X,D)}^{orb}$, we note that by Corollary $2.10$ of \cite{logroot}, it is equivalent to consider an analogous integral on the corresponding moduli space of naive log maps $\mathscr{M}^{naive}(X,\beta)$. By the main theorem of \cite{mirrcomp}, the contribution to the invariant of interest coming from the locus of curves with log enhancements vanishes. Hence, any non-zero contribution must come from the locus of curves which do not admit log enhancements. To produce one such curve, we take the unique curve $C_1$ which contributes to $N_{p_1,p}^{2L-E}$ above, glue a chain of rational curves $C_0\cup C_2$ to the point of $C_1$ associated with the contact order $p$ point, and equip $C_0$ with an extra marked point. Calling the resulting curve $C$, we extend $f|_{C_1}: C_1\rightarrow X$ to a map $f: C \rightarrow X$ by contracting $C_0$, and letting $C_2$ map isomorphically onto the strict transform of the line between $x$ and the zero stratum $D_2\cap D_3$. 

There is a unique way of extending this map to naive log map with contact orders $p_1$ and $p_2$ and $0$ at marked points, but no logarithmic enhancement exists. Indeed, the obstruction to the log enhancement can be seen tropically; the tropical curve would have three vertices, two of which map to the zero cone and are connected to the third vertex $v_0$ lying in the cone spanned by $D_2$ and $D_3$ by a distinct edge. The slopes of these edges are given by $2D_2 + D_3$ and $D_2 + D_3$, hence cannot intersect in the interior of $\pmb\sigma(v_0)$, which obstructs the required tropical curve from existing. On the other hand, using the fact that both the linear functions $D_3^*$ and $D_2^*$ do not restrict to zero on each edge, it is straightforward to show that a naive log enhancement exists, see \cite{logroot} Example $3.10$ for a similar example, with Example $3.13$ of loc. cit. explaining the obstruction to a log enhacement.

 By recalling no irreducible curve satisfies the required constraints, combined with considering the restriction on combinatorial types imposed by the curve class and contact orders of marked points, one can show that all curves contributing to the naive invariant have combinatorial type as above. Using the fact that the conic is isolated among solutions to the point constraint in its relevant moduli space, it is straightforward to show that $C$ is the unique curve potentially contributing to the invariant. 

To prove that $C$ contributes to the orbifold trace, we need to show the curve is in addition unobstructed. By the description of the obstruction theory for naive log maps described in Section $2.2$ of \cite{logroot}, it suffices to show that $\mathfrak{M}^{naive}(\mathcal{X},\beta)$ is locally smooth and dimension $0$ around the image of $[C] \in \mathfrak{M}^{naive}(\mathcal{X},\beta)(\kk)$ and $H^1(C,f^*T_X^{log}) = 0$. If these conditions hold, then locally around $[C]$, $\mathscr{M}^{naive}(X,\beta)$ is smooth and $2$ dimensional by Riemann-Roch. Since $2$ is the expected dimension of the moduli space, the map $C \rightarrow X$ would be unobstructed.

 For the first property, let $\beta_i$ be the tropical type of log stable map to $(X,D_i)$ induced by the naive tropical type $\beta$. By the classification of strata of $\mathfrak{M}^{naive}(\mathcal{X},\beta)$ as fiber products of locally closed strata of maps to $\mathfrak{M}(\mathcal{A}_{\NN},\beta_i)$ over the moduli stack of prestable curves, we derive the stratum of interest which contains $[C]$ is smooth and dimension $0$. Moreover, since the nodes of $\beta_i$ must smooth simultaneously for $i=2,3$, the only possible distinct stratum whose closure contains the stratum of interest is the non-degenerate stratum . However, since no curve in the stratum of interest has a log enhancement, we conclude that the stratum is open. Since the virtual dimension of $\mathfrak{M}^{naive}(\mathcal{X},\beta)$ is $0$, we conclude that $\mathfrak{M}^{naive}(\mathcal{X},\beta)$ is smooth and of dimension $0$ locally around the image of $[C]$. 

Turning now to the computation of $H^1(C_i,f^*|_{C_i}T_X^{log})$, since $C_0$ is contracted, clearly $f|_{C_0}^*T_X^{log} = \mathcal{O}_{C_0}^2$. Moreover, since $C_1$ contributes to the naive curve count $N_{p_1,p}^{2L-E} = \eta(p_1,p,2L-E)$ from \cite{archmirror} Definition $1.1$, by loc. cit. Proposition $3.13$, we have $f^*|{C_1}T_X^{log} = \mathcal{O}_{C_1}^2$. For $C_2$, note that since $c_1(T_X^{log})\cdot(L-E) = 0$, Riemann-Roch gives 
\[dim_{\kk}\text{ }H^0(f_{C_2}^*T_X^{log}) - dim_{\kk}\text{ }H^1(f_{C_2}^*T_X^{log}) = 2.\]
Since the only deformations of the log map come from infinitesimal automorphisms of the domain curve which fix the unique marked point, we have $H^0(f_{C_2}^*T_X^{log}) = \kk^2$, implying $H^1(f_{C_2}^*T_X^{log})=0 $


With these facts in mind, we now consider the following normalization sequence:

\[0 \rightarrow f^*(T_X^{log}) \rightarrow \oplus_i f|_{C_i}^*(T_X^{log}) \rightarrow \kk_{e_1}^2 \oplus \kk_{e_2}^2\rightarrow 0\]
As usual, the middle term above is considered a sheaf on $C$ via pushforward, and $e_i$ refer to the nodes of $C$. Taking the long exact sequence in cohomology gives the sequence:

\[H^0(\oplus_i f|_{C_i}^*(T_X^{log})) \rightarrow \kk_{e_1}^2\oplus \kk_{e_2}^2 \rightarrow H^1(f^*(T_X^{log})) \rightarrow H^1(\oplus_i f|_{C_i}^*(T_X^{log}))\]
By our previous computation of $f|_{C_i}^*(T_X^{log})$ for each irreducible component $C_i$, the rightmost term vanishes. Hence, to show $H^1(f^*(T_X^{log})) = 0$, it suffices to show that the left most arrow in the sequence above is surjective. 

To do this, we consider the image of each component of $\oplus_i f|_{C_i}^*(T_X^{log})$ in $\kk_{e_1}^2\oplus \kk_{e_2}^2$ for $i = 0,1$. First, since $f|_{C_0}^*T_X^{log} = \mathcal{O}^2_{C_0}$, we have the restriction map is given by the diagonal $\Delta: \kk^2 \rightarrow \kk_{e_1}^2\oplus\kk_{e_2}^2$. For $C_1$, since $f^*|_{C_1}T_X^{log}$ is the trivial bundle, it is in particular globally generated. Hence, the restriction map to $\kk_{e_1}^2$ is surjective. Since the intersection of the images of these linear maps has trivial intersection, it follows that the sum of the two maps is surjective, and hence the leftmost morphism in the exact sequence above is surjective as required. 

\textbf{Case $2$: Non-equality due to refined virtual class}:

In this second example, we consider $p= D_2 + D_3$, and we are interested in the structure constant $(\vartheta_{p}^{orb})^2[\vartheta_{p}^{orb}z^{L-E}]$. The intrinsic mirror structure constant is easily computed to vanish, once again using a broken line argument.

 To compute the orbifold structure constant, we again use a naive log/orbifold correspondence, this time coming from Theorem $5.1$ of \cite{logroot2} to account for the high age/negative contact orders which must be present. In the naive log setting, letting $\beta$ be the combinatorial type of twisted map contributing to the invariant of interest, it is straightforward to see that the domain of such maps consists of two components $C_1\cup C_2$ joined along a node $q$ with $C_2$ containing three special points, with the map mapping $C_1$ isomorphically onto the strict transform of the line considered in case $1$ and $C_2$ contracted. Each such stable map admits a naive log enhacement of the correct type. In fact, unlike in case $1$, each admits a logarithmic enhacement. A map of this type is thus determined by a choice of cross-ratio of the $4$ special points on $C_2$, hence as stacks, we have $\mathscr{M}^{naive}(X,\beta) \cong \overline{\mathscr{M}}_{0,4} \cong \mathbb{P}^1$. However, by equation \ref{vdimorb}, we have $vdim\text{ }\mathscr{M}^{naive}(X,\beta) = 0$. By the naive log/orbifold correspondence, the structure constant in question can be computed as the degree of this virtual cycle. 

To compute the virtual cycle, observe that the map $\mathscr{M}^{naive}(X,\beta) \rightarrow \mathfrak{M}^{naive}(\mathcal{X},\beta)$ has image entirely contained in a stratum associated with the realizable tropical type $\tau$ described above. By essentially the same calculation which appears in example $1.12$ of \cite{logroot2}, we have 
\[[\mathscr{M}^{naive}(X,\beta)]^{ref} = [\mathscr{M}^{naive}(X,\beta)]\cap (-\psi_{q \in C_1} - \psi_{q \in C_2})\]
where the cotangent lines above are associated with the two flags of $G_{\tau}$. The component of the universal family over $\mathscr{M}^{naive}(X,\beta)$ associated with $C_1$ is a trivial family of once-marked genus $0$ curves, hence $[\mathscr{M}^{naive}(X,\beta)]\cap (-\psi_{q \in C_1})= 0$. On the other hand, under the isomorphism $\mathscr{M}^{naive}(X,\beta) \cong \overline{\mathscr{M}}_{0,4}$, $\psi_{q \in C_2}$ is the pullback of the associated cotangent line on $\overline{\mathscr{M}}_{0,4}$. Hence, $[\mathscr{M}^{naive}(X,\beta)]\cap (-\psi_{q \in C_2})= -1$. In summary, $(\vartheta_{p}^{orb})^2[\vartheta_{p}^{orb}z^{L-E}] = -1$. We contrast this calculation with the fact that all logarithmic structure constants in this setting are non-negative.

We note however that we nonetheless have a $\kk[NE(X)]$ algebra isomorphism $f\colon R_{(X,D)}^{orb} \rightarrow R_{(X,D)}$, with $f(\vartheta_{D_i}^{orb}) = \vartheta_{D_i}$ for all $i$. Indeed, by \cite{int_mirror} Example $1.21$, we have the presentation:

\[R_{(X,D)} \cong \kk[NE(X)][\vartheta_{D_1},\vartheta_{D_2},\vartheta_{D_3}]/(\vartheta_{D_1}\vartheta_{D_2}\vartheta_{D_3} - z^L - \vartheta_{D_1}z^{L-E})\]
In order to show that the assignment $f(\vartheta_{D_i}^{orb}) = \vartheta_{D_i}$ extends to an isomorphism, just as in the calculation of $R_{(X,D)}$ from loc. cit. it suffices to compute $\vartheta^{orb}_{D_1}\vartheta^{orb}_{D_2}\vartheta^{orb}_{D_3}$. The problems highlighted above in the relevant log and orbifold invariants do not occur for calculating the relevant naive/orbifold invariants, and we derive the analogous expression for the triple product. Hence $R_{(X,D)}^{orb}$ has exactly the same presentation as $R_{(X,D)}$. In particular, in this example, the log and orbifold bases can be expressed as $\kk[NE(X)]$ linear combinations of each other.

\section{Log Gromov-Witten theory of root stacks}

At various stages of the arguments in this paper, we will use a generalization of the main result from \cite{pbirinv} to the setting of log \'etale modifications given by root stacks. A technical issue occurs when considering this situation as the theory of log stable maps to a log DM stack has not yet been fully developed. However, when the orbifold structure arises via a log \'etale modfication of a log smooth scheme, it is relatively straightforward to construct the relevant virtually smooth moduli stacks which satisfy the required properties. This has been constructed by Battistella, Naibjou and Ranganathan in \cite{logroot2} assuming only specified global contact orders. In this section, we will provide the construction when we assume the tropical types are realizable, from which the necessary generalization of the main result of \cite{pbirinv} will follow in a straightforward manner. For future purposes, we will work in a slightly more general context than is strictly necessary for the purpose of this paper.

Suppose we have a log smooth morphism $X \rightarrow B$ with $B$ log smooth, and a morphism $\mathcal{A}_X \rightarrow \mathcal{A}_B$ making the obvious square commute. We define the relative Artin fan $\mathcal{X}:= B\times_{\mathcal{A}_B}\mathcal{A}_X$.  Suppose in addition we have a root stack $\tilde{X} \rightarrow X$ realized as a log \'etale modification. To construct a moduli space of twisted log stable maps to $\tilde{X}$ over $B$, we observe that we can construct the root stack $\tilde{X}$ as a fiber product:

\[\begin{tikzcd}
\tilde{X} \arrow{r} \arrow{d} & \tilde{\mathcal{X}} \arrow{d}\arrow{r} & \mathcal{A}_{\tilde{X}}\arrow{d} \\
X \arrow{r} \arrow{rd} & \mathcal{X} \arrow{r}\arrow{d} & \mathcal{A}_X\arrow{d}\\
&B \arrow{r} &\mathcal{A}_B
\end{tikzcd}.\]
The bottom commuting square is the cartesian square defining $\mathcal{X}$ and the top right vertical morphism is a log \'etale modification induced by a lattice refinement. More precisely, the morphism $\mathcal{A}_{\tilde{X}} \rightarrow \mathcal{A}_X$ is \'etale locally on $\mathcal{X}$ induced by a lattice refinement of Artin cones $\mathcal{A}_{\tilde{\sigma}} \rightarrow \mathcal{A}_\sigma$. Letting $\Gamma_\tau \rightarrow \Sigma(X)$ be a family of punctured tropical maps, we form the fiber product:

\[\begin{tikzcd}
\tilde{\Gamma}_\tau^\circ \arrow{r}\arrow{d} & \Sigma(\tilde{X}) \arrow{d}\\
\Gamma_\tau^\circ \arrow{r} & \Sigma(X)
\end{tikzcd}.\]

The induced morphism $\tilde{\Gamma}_\tau \rightarrow \Gamma_\tau$ is a lattice refinement of cone complexes. By pushing forward the refinement to $\tau$, we produce a cone $\gamma$. Moreover since the all the vertical maps on cones are induced by lattice refinements, we have $\gamma\otimes \mathbb{Q} \cong \tau \otimes \mathbb{Q}$. After pulling back $\tilde{\Gamma}_\tau \rightarrow \tau$ along $\gamma \rightarrow \tau$, by the same proof of Proposition $4.1$ of \cite{pbirinv}, we produce a family of punctured tropical maps $\tilde{\Gamma}_\gamma \rightarrow \Sigma(\tilde{X})$. Moreover, the previous family of maps is universal in the sense that any family of maps $\Gamma'_\omega \rightarrow \Sigma(X)$ induced by pulling back $\Gamma_\tau$ along a map $\omega \rightarrow \tau$ which factors through the refinement $\Sigma(\tilde{X}) \rightarrow \Sigma(X)$ is induced by a unique morphism $\omega \rightarrow \gamma$ which factors $\omega \rightarrow \tau$. Using this fact, it is straightforward that refinements $\gamma \rightarrow \tau$ glue to a determine a log \'etale modification of $\mathcal{A}_{\mathfrak{M}(\mathcal{X})}$. We pullback this modification along the strict morphism $\mathfrak{M}(\mathcal{X},\tau) \rightarrow \mathcal{A}_{\mathfrak{M}(\mathcal{X})}$ to produce a log stack $\mathfrak{M}_{\tilde{\mathcal{X}},\tau'} $. Pulling $\mathfrak{M}_{\tilde{\mathcal{X}},\tilde{\tau}} \rightarrow \mathfrak{M}(\mathcal{X},\tau)$ along the strict map $\mathscr{M}(X,\tau) \rightarrow \mathfrak{M}(\mathcal{X},\tau)$ produces a log DM stack $\mathscr{M}_{\tilde{X},\tilde{\tau}}$, which we record in the fiber diagram below:

\begin{equation}\label{twistuniprop}
\begin{tikzcd}
\mathscr{M}_{\tilde{X},\tilde{\tau}} \arrow{r} \arrow{d} & \mathfrak{M}_{\tilde{\mathcal{X}},\tilde{\tau}} \arrow{d}\\
\mathscr{M}(X,\tau) \arrow{r} & \mathfrak{M}(\mathcal{X},\tau).
\end{tikzcd}
\end{equation}

We claim that the stack $\mathscr{M}_{\tilde{X},\tau'}$ represents a moduli stack of twisted punctured log maps to $\tilde{X}$, a statement which we will clarify shortly. Before this however, we recall a collection of definitions and results from \cite{ltcur}.

\begin{definition}
\begin{enumerate}
\item A morphism $\mathcal{M}\rightarrow \mathcal{N}$ of locally free log structures on a DM stack $X$ is called \emph{simple} if for every geometric point $p \rightarrow X$, the monoids $\overline{\mathcal{M}}_p$ and $\mathcal{\mathcal{N}}_p$ have the same rank, the morphism $\phi: \overline{\mathcal{M}}_p \rightarrow \overline{\mathcal{N}}_p$ is injective, and for every irreducible element $f \in \overline{\mathcal{N}}_p$ there exists an irreducible element $g \in \overline{\mathcal{M}}_p$ and a positive integer $n$ such that $\phi(g) = nf$.

\item A log twisted curve is the data $(C/S,\sigma_i,a_i,\mathcal{M}_S^{min} \rightarrow \mathcal{M}_S)$ with $C/S$ a stable curve, $\sigma_i: S \rightarrow C$ sections of the curve, $a_i$ positive integers indexed by the sections, and $\mathcal{M}_S^{min} \rightarrow \mathcal{M}_S$ a simple morphism of fine and saturated log structures on $S$ with $\mathcal{M}^{min}_S$ the minimal log structure induced by the classifying morphism $S \rightarrow \overline{\mathscr{M}}_{g,n}$.

\item A twisted log curve is a prestable twisted curve $\mathcal{C}/S$ with log structures on $\mathcal{C}$ and $S$ such that $\mathcal{C} \rightarrow S$ is log smooth. 
\end{enumerate}
\end{definition}

Note that the tropicalization of a simple log morphism $(S,\mathcal{N}_S) \rightarrow (S,\mathcal{M}_S)$ is a morphism of cone complexes which upon restriction to cones of the domain are given by a lattice refinements. We call the lattice refinements arising in this way simple lattice refinements. Both twisted curves and the enriched curves defined above over some base scheme $S$ are objects in naturally defined groupoids. Theorem $1.9$ of \cite{ltcur} states the following:

\begin{proposition}\label{logtw}
For any scheme $S$, there is a natural equivalence of groupoids between the groupoid of $n$-pointed twisted curves over $S$ and the groupoid of log twisted $n$-pointed curves over $S$. Moreover, this equivalence is compatible with base change $S' \rightarrow S$.
\end{proposition}

In order to frame the definition of a twisted log curve in a way which will be useful for us, we briefly recall the mechanics of the above equivalence. One first shows that every family of twisted curves $\mathcal{C} \rightarrow S$ admits a canonical log enhancement, producing a twisted log curve over a log scheme $(S,\mathcal{M}_S')$, and the log structure on $\mathcal{C}$ pushes forward to the coarse curve $C$. By the universal property satisfied by basic log curves, the log curve $C/S$ pulls back from a basic log curve $C^{min}/S^{min}$. The resulting morphism of log structures on $S$ is a simple morphism of log structures. The tuple of data given by the orders of twisting at marked points, the prestable curve $C/S$ and the morphism $\mathcal{M}_{S^{min}} \rightarrow \mathcal{M}_{S}$ gives a family of log twisted curves. Conversely, every log twisted curve is shown is shown to come from a unique such orbifold enhancement of the underlying family of stable curve. Moreover, the canonical log enhancement of a twisted curve satisfies a similar universal property familiar from the theory of ordinary stable curves, namely any other log enhancement of the twisted curve $\mathcal{C} \rightarrow S$ is realized by an appropriate log morphism on the base induced by a unique morphism of log structures on $S$. As a result, we can think of a twisted log curve both as in the definition as above, or as the data of a log twisted curve together with a morphism of log schemes $(S,\mathcal{M}_S)\rightarrow (S,\mathcal{M}_S')$, with $\mathcal{M}_S'$ as above.

Twisted log curves arise via tropical considerations in the following way for us: Suppose that we have a family of prestable log curve $C/S$, and for simplicity we assume the dual graph over all schematic points is given by $G_\tau$. Tropicalization gives a family of tropical curves $\Gamma \rightarrow \Sigma(S) = \sigma$, which is induced by a morphism $\sigma \rightarrow \tau = \mathbb{R}_{\ge 0}^{|E(G_\tau)|}$. Suppose we have a lattice refinement $\Gamma' \rightarrow \Gamma$ such that for all cones $\omega$ of $\Gamma'$, the map $\omega_\NN \rightarrow \sigma_\NN$ is surjective. Such a modification is determined by numbers $k_e$ for each edge and leg given by $|coker(\sigma_{e',\NN}^{gp} \rightarrow \sigma_{e,\NN}^{gp})|$ for each edge and leg $e$ refined by an edge/leg $e'$ of $\Gamma'$. Since for all $e \in E(G_\tau)$ and all integral tropical curves in the family $\Gamma$, the length of the edge $e$ must be divisible by $k_e$, the cone map $\sigma \rightarrow \tau = \mathbb{R}_{\ge 0}^{|E(G_\tau)|}$ will factor through a simple lattice refinement $\tau'$ of $\tau$ determined by the numbers $k_e$. Moreover, the refinement $\Gamma' \rightarrow \Gamma$ is induced by pulling back a refinement $\Gamma'_{\tau'}\rightarrow \Gamma_{\tau'}$ of the pullback of the universal curve to $\tau'$.


Now recall from Lemma \ref{equiv} that the morphism $\sigma \rightarrow \tau'$ corresponds to a morphism of Artin fans $\mathcal{A}_{\sigma} \rightarrow \mathcal{A}_{\tau'}$, and by composing the strict morphism $(S,\mathcal{M}_{S}) \rightarrow \mathcal{A}_{\sigma}$ with this map and pulling back the log structure, we produce a log structure we call $\mathcal{M}^{tw}$. Similarly, the morphism $\tau' \rightarrow \tau$ induces a simple morphism of log structures $\mathcal{M}_S^{min} \rightarrow \mathcal{M}_S^{tw}$ with $\mathcal{M}_S^{min}$ the basic log structure on $S$. Hence we find the data of a log twisted curve together with a map $(S,\mathcal{M}_S) \rightarrow (S,\mathcal{M}_S^{tw})$. By the remarks in the paragraph following Proposition \ref{logtw}, this gives by a twisted log curve $\mathcal{C}$ with coarse curve $C$. Moreover, the morphism to the coarse curve tropicalizes to $\Gamma' \rightarrow \Gamma$. 


We now give a lemma describing how puncturings of twisted log curve $\mathcal{C}$ may be produced from puncturing of the coarse log curve $C$.

\begin{lemma}
Suppose $C \rightarrow S$ is a log curve, $\rho: \mathcal{C} \rightarrow C$ is a twisted log curve over $C$ and $C^\circ\rightarrow C$ is a puncturing of $C$. Then $\mathcal{C}^\circ := \mathcal{C}\times_{C}C^\circ \rightarrow \mathcal{C}$ is a puncturing of the twisted log curve $\mathcal{C}$.
\end{lemma}

\begin{proof}
First observe that the map $\rho^*\mathcal{M}_C \rightarrow \mathcal{M}_{\mathcal{C}}$ is an integral morphism of log structures, hence $\mathcal{C} \rightarrow C$ is integral. Thus, the coherent log structure $\mathcal{C}^\circ$ is also fine. To verify the second condition of Definition \ref{defpunc}, we note that for $p \in \mathcal{C}$ a punctured point, and $s \in \mathcal{M}_{\mathcal{C}^\circ,p}$, we have $s \notin \mathcal{M}_{\mathcal{C},p}$ if and only if its image $\overline{s} \in \overline{\mathcal{M}}_{\mathcal{C}^\circ,p}$ is not contained in $\overline{\mathcal{M}}_{\mathcal{C},p} \subset \overline{\mathcal{M}}_{\mathcal{C}^\circ,p}$. Since the log structure on $\mathcal{C}^\circ$ is fine, we can compute the stalk $\overline{\mathcal{M}}_{\mathcal{C}^\circ,p}$ in the category of integral monoids. Letting $Q_p = \overline{\mathcal{M}}_{C^\circ,p} \subset \overline{\mathcal{M}}_{S,\pi(p)}\oplus \ZZ$, $\overline{\mathcal{M}}_{\mathcal{C},p} = \overline{\mathcal{M}}_{S,\pi(p)} \oplus \NN$, the pushout is given by the smallest submonoid of $\overline{\mathcal{M}}_{S,\pi(s)} \oplus \ZZ$ containing $\overline{\mathcal{M}}_{S,\pi(s)} \oplus \NN$ and $i(Q)$, with $i: \overline{\mathcal{M}}_{C,p}^{gp} \rightarrow \overline{\mathcal{M}}_{\mathcal{C},p}^{gp}$. In particular, if $\overline{s} \notin \overline{\mathcal{M}}_{\mathcal{C},p}$, we can factor $\overline{s}$ as $\overline{s} = \overline{s}_1 + \overline{s}_2$ with $\overline{s}_1 \in \overline{\mathcal{M}}_{s,\pi(s)} \oplus \NN$ and $\overline{s}_2 = i(\overline{t})$ for $\overline{t} \in \overline{\mathcal{M}}_{C^\circ,p} \setminus \overline{\mathcal{M}}_{C,p}$. By picking lifts of $s_1 \in \mathcal{M}_{\mathcal{C},p}$ and $t \in \mathcal{M}_{C^\circ,p}$, we have $s_1 + s_2$ differs from $s$ only by a unit. In particular, $\alpha_{\mathcal{C}^\circ}(s) = 0$ if and only if $\alpha_{C^\circ}(t) = 0$, which holds since $C^\circ$ is a punctured log curve. Similarly, writing $s = (e_1,e_2)$ and $s_i = (e_{i1},e_{i2})$, then $e_1$ is equal to $e_{11}+e_{21}$ up to a unit in $\mathcal{O}_{S,\pi(p)}^\times$. Since $\alpha_{S}(e_{21}) = 0$ as $C^\circ$ is a punctured log curve, we have $\alpha_S(e_1) = 0$, as required. 
\end{proof}

We now proceed to a proof of the representability claim:

\begin{proposition}\label{repfun}
Let $\mathscr{M}(\tilde{X}/B,\tilde{\tau})$ be the fibered category defined over log schemes such that the objects of the groupoid over $S$ are the following data:

\begin{enumerate}
\item A twisted log curves $\mathcal{C} \rightarrow S$ over an ordinary log curve $C \rightarrow S$, in particular a log structure $\mathcal{M}_S$ on $S$.

\item A punctured log map $C^\circ \rightarrow X$ marked by the type $\tau$.

\item A commuting diagram of logarithmic morphisms:

\begin{equation}\label{maproot}
\begin{tikzcd}
\mathcal{C}^\circ \arrow{r} \arrow[d,"\rho"] & \tilde{X}\arrow[d,"c"]\\
C^\circ \arrow{r}\arrow{d} & X\arrow{d}\\
(S,\mathcal{M}_S) \arrow{r} & B
\end{tikzcd}.
\end{equation}
We further require the top horizontal morphism to be representable.


 \end{enumerate} 
Then the fibered category of $\mathscr{M}_{\tilde{X},\tilde{\tau}}$ over fine and saturated log schemes is isomorphic to $\mathscr{M}(\tilde{X}/B,\tilde{\tau})$. 

\end{proposition}

\begin{proof}
We recall that $\mathscr{M}_{\tilde{X},\tilde{\tau}}$ is defined as the fiber product given in Display \ref{twistuniprop}. Hence, to prove the isomorphism, we must show that $\mathscr{M}(\tilde{X}/B,\tilde{\tau})$ satisfies the the same universal property as $\mathscr{M}_{\tilde{X}/B,\tilde{\tau}}$. Clearly we have a morphism $\mathscr{M}(\tilde{X}/B,\tilde{\tau}) \rightarrow \mathscr{M}(X/B,\tau)$ by taking the map associated with $C^\circ \rightarrow \mathcal{X}$. By definition, the tropical type of every twisted log map $\mathcal{C}^\circ \rightarrow \tilde{X}$ associated with a geometric point of $\mathscr{M}(\tilde{X}/B,\tilde{\tau})$ is a tropical lift of a log stable map $C\rightarrow X$. Hence, the proof of Proposition $5.4$ of \cite{pbirinv} applies in this setting to give a strict and idealized strict morphism $\mathscr{M}(\tilde{X}/B,\tilde{\tau}) \rightarrow \mathfrak{M}_{\tilde{\mathcal{X}},\tilde{\tau}} $ which factors the natural morphism $\mathscr{M}(\tilde{X}/B,\tilde{\tau}) \rightarrow \mathfrak{M}(\mathcal{X}/B,\tau)$. 

In order to show the resulting square is cartesian, we must first construct a morphism of stacks $\mathfrak{M}_{\tilde{\mathcal{X}},\tilde{\tau}} \rightarrow \mathfrak{M}(\tilde{\mathcal{X}}/B,\tilde{\tau})$ which factors $\mathscr{M}(\tilde{X}/B,\tilde{\tau}) \rightarrow \mathfrak{M}(\tilde{\mathcal{X}}/B,\tilde{\tau})$, with the latter stack defined analogously to $\mathscr{M}(\tilde{X}/B,\tilde{\tau})$. The construction works essentially unchanged from Section $5.1$ of \cite{pbirinv}, although is simpler in this setting due to the uniqueness of a tropical lift of a tropical type. 

Suppose we have a fine and saturated scheme $S$, together with an $S$-valued object of $\mathfrak{M}_{\tilde{\mathcal{X}},\tilde{\tau}}$, giving a morphism $C^\circ/S \rightarrow X$. By the same reduction that appears in Section $5.1$ of \cite{pbirinv}, we may assume the map $S \rightarrow \mathcal{A}_{\mathfrak{M}_{\tilde{X},\tau}}$ factors through an inclusion of an idealized Artin cone $\mathcal{A}_{\tau',I} \subset \mathcal{A}_{\mathfrak{M}_{\tilde{X},\tau}}$, hence $\Sigma(S) = \tau'$. Associated to $\tau'$ is a universal family of punctured tropical curves $\tilde{\Gamma}^\circ_{\tau'} \rightarrow \tilde{X}$ over the cone $\tau'$, and a map of cone complexes $\tilde{\Gamma}^\circ_{\tau'} \rightarrow \Gamma_{\tau'}$ given cone wise by a lattice refinement. Since $C^\circ/S \rightarrow X$ is induced by a morphism $S \rightarrow \mathfrak{M}_{\tilde{\mathcal{X}},\tau}$, we have $\Sigma(C^\circ) \rightarrow \Sigma(X)$ restricted to cones associated with vertices of $\tau$ factors through the lattice refinement $\Sigma(\tilde{X}) \rightarrow \Sigma(X)$. 

In order to lift extend the tropical lift to edges and legs, we must refine the lattices of the edges and legs of $\Gamma_\tau$. To produce this refinement, note that since the integral points of the maximal cones of $\Gamma_{\tau'}^\circ$ naturally embed into the lattice $\tau_\NN^{'gp} \times \ZZ$, and the refinement morphism restricts to an isomorphism on the submodule $\tau_{\NN}^{gp'}$, the refinement restricted to a given maximal cone associated with an edge or leg $e$ must be of the form $\tau_\NN^{'gp} \times r_e\ZZ \subset \tau_\NN^{'gp}\times \ZZ$ for some natural number $r_e \in \NN_{\ge 0}$. Hence, this lattice refinement induces a twisted log curve $\mathcal{C} \rightarrow C$. The tropical lift $\tilde{\Gamma} \rightarrow \Sigma(\tilde{X})$ of $\Gamma \rightarrow \Sigma(X)$ guarantees there exists a map $\tilde{f}: \mathcal{C}^\circ \rightarrow \tilde{\mathcal{X}}$ lifting the map $C^\circ \rightarrow \mathcal{X}$. 

To see that $\tilde{f}$ is representable, note that this is true if the map induced on stabilizer groups of geometric points is injective. Thus, let $q' \in \mathcal{C}(\kk)$ be a geometric point contained in either a nodal or marked section of $\mathcal{C}$, and suppose $Stab(q') \cong \mu_r$. In particular, the induced morphism $m_q: q' \rightarrow \rho(q') := q$ is a generalized root stack of the log point $q$, introduced in the last paragraph of Section $3.1$. Letting $\omega' = \Sigma(q')$ and $\omega = \Sigma(q)$, this map is induced by a lattice refinement $\omega^{'gp}_\NN \rightarrow \omega^{gp}_\NN$, which with a slight abuse of notation we will refer to as $m_q$ as well. Moreover, we have an induced map $B\mu_r \rightarrow \tilde{\mathcal{X}}$ which lifts a map $Spec\text{ }\kk \rightarrow \mathcal{X}$. Letting $p = c\circ\tilde{f}(q) \in \mathcal{X}(\kk)$, the morphism on stabilizer groups $Stab(q) \rightarrow Stab(p)$ induced by $c\circ f: \mathcal{C} \rightarrow \mathcal{X}$ must be the zero map, hence the morphism $B\mu_r \rightarrow \mathcal{X}$ must factor through a geometric point. Since $\tilde{\mathcal{X}}$ is a generalized root stack over $\mathcal{X}$ constructed via a fiber diagram in algebraic stacks, the morphism $B\mu_r \rightarrow \tilde{\mathcal{X}}$ must factor through a root stack of the log point $p$, denoted by $m_p: p' \rightarrow p$. Letting $\Sigma(p') = \sigma'$ and $\Sigma(p) = \sigma$, the morphism $m_p$ is induced by a lattice refinement $\sigma^{'gp}_\NN \rightarrow \sigma^{gp}_\NN$, which again with a slight abuse of notation we will refer to as $m_p$ as well.

Since $p' \rightarrow \tilde{\mathcal{X}}$ is representable, in order to show that $\tilde{f}$ is representable, it suffices to show that $\tilde{f}|_{q'}: B\mu_r \rightarrow p'$ is representable. Recall that if $m_p: p' \rightarrow p$ is a generalized root stack realized as a log \'etale modification of a log point $p$ with $\Sigma(p) = \sigma$ and $\Sigma(p') = \sigma'$ induced by a lattice refinement $m_p$, then $Stab(p') = \sigma^{gp}_\NN/m_p(\sigma^{'gp}_\NN)$. Since $m_p: p' \rightarrow p$ and $m_q: q' \rightarrow q$ are log monomorphisms, $f: q \rightarrow p$ determines $\tilde{f}: q' \rightarrow p'$, we must have $\Sigma(f)^{gp}_\NN((\Sigma(m_q)(\omega^{'gp}_\NN)) \subset m_p(\sigma^{'gp}_\NN)$ and the morphism on stabilizers is induced by $\omega_\NN^{gp} \rightarrow \sigma_\NN^{gp}$. This induced map is injective only if 
\[\Sigma(f)^{gp}_\NN((m_q(\omega^{'gp}_\NN))) = \Sigma(f)^{gp}_\NN(\omega_\NN^{gp})\cap m_p(\sigma^{'gp}_\NN).\]
By the minimal choice of lattice refinement at each leg and edge of $G_\tau$ taken to allow a lift of $\tilde{\Gamma}\rightarrow \Gamma$ along the refinement $\Sigma(\tilde{X}) \rightarrow \Sigma(X)$, the representability of $\tilde{f}$ follows.

In total we find an $S$-valued object of $\mathfrak{M}(\tilde{\mathcal{X}}/B,\tilde{\tau})$. The proofs of Propositions $5.7$ and $5.8$ of \cite{pbirinv} apply to show there is a morphism $\mathfrak{M}_{\tilde{\mathcal{X}},\tau}  \rightarrow \mathfrak{M}(\tilde{\mathcal{X}}/B,\tilde{\tau})$ which factors $\mathscr{M}(\tilde{X}/B,\tilde{\tau}) \rightarrow \mathfrak{M}(\tilde{\mathcal{X}}/B,\tilde{\tau})$. Finally the proof of Proposition $6.1$ of \cite{pbirinv} works to show that $\mathscr{M}(\tilde{X}/B,\tilde{\tau})$ satisfies the required universal property. 

\end{proof}

The morphism $\mathscr{M}(\tilde{X}/b,\tilde{\tau}) \rightarrow \mathfrak{M}_{\tilde{\mathcal{X}},\tilde{\tau}}$ can be equipped with the pulled back obstruction theory from $\mathscr{M}(X/B,\tau) \rightarrow \mathfrak{M}(\mathcal{X}/B,\tau)$, which may equivalently be given by the dual of:
\[ R\pi_*f^*T_{\tilde{X}/B}^{log} = R\pi_*f^*c^*T_{X/B}^{log}\]
The equality holds because the map to the coarse moduli space $c:\tilde{X} \rightarrow X$ is log \'etale. Virtual pullback of the fundamental class of $\mathfrak{M}_{\tilde{\mathcal{X}},\tau}$ is the sense of \cite{vpull} equips $\mathscr{M}(\tilde{X}/B,\tilde{\tau})$ with a virtual class $[\mathscr{M}(\tilde{X}/B,\tilde{\tau})]^{vir} \in A_{exp\text{ }dim}(\mathscr{M}(\tilde{X}/B,\tilde{\tau}))$. By the same arguments as in \cite{pbirinv}, the virtual classes are related by the following formula:

\begin{corollary}\label{birinvr}
Letting $r = |coker(\tilde{\tau}^{gp} \rightarrow \tau^{gp})|$, and $st: \mathscr{M}(\tilde{X},\tilde{\tau}) \rightarrow \mathscr{M}(X/B,\tau)$ the map which takes a twisted log map to the corresponding map on coarse spaces, we have the equality:
\[st_*([\mathscr{M}(\tilde{X}/B,\tilde{\tau})]^{vir}) = \frac{1}{r} [\mathscr{M}(X/B,\tau)]^{vir}.\]
\end{corollary}

\begin{remark}\label{idloget}
Battistella, Nabijou and Ranganathan constructed in \cite{logroot2} moduli stacks $\mathscr{M}^{naive}(\tilde{X},\beta')$ and $\mathfrak{M}^{naive}(\tilde{\mathcal{X}},\beta')$ of naive twisted log maps to root stacks for non-realizable tropical types $\beta'$ which specify only (gerby) contact orders of marked points, together with a perfect obstruction theory for the morphism $\mathscr{M}^{naive}(\tilde{X},\beta')\rightarrow \mathfrak{M}^{naive}(\tilde{\mathcal{X}},\beta')$. Additionally, the forgetful map $\mathfrak{M}^{naive}(\tilde{\mathcal{X}},\beta') \rightarrow \mathfrak{M}^{tw}$ remembering only the domain twisted curve is idealized log \'etale, with the target equipped with the trivial idealized log structure. Moreover, assuming $\beta'$ is slope sensitive, the reduction of the stack $\mathfrak{M}^{naive}(\tilde{\mathcal{X}},\beta')_{red}$ is stratified by components corresponding to realizable tropical types, which recover the log stacks defined above. With the previous assumption, we will omit the superscript and call the resulting stacks $\mathscr{M}(\tilde{X},\beta')$ and $\mathfrak{M}(\tilde{\mathcal{X}},\beta')$. 

\end{remark}

%
%

In the following, we note how the above corollary determines how punctured Gromov-Witten invariants behave under a log \'etale base change, which may be useful in studying semistable degenerations of varieties:

Suppose $X \rightarrow \mathbb{A}^1$ is a log smooth morphism, with $\mathbb{A}^1$ carrying its usual toric log structure, and consider the log \'etale morphism $z^k: \mathbb{A}^1\rightarrow \mathbb{A}^1$ induced by regular function $z^k \in \kk[z]$. We denote by $B$ and $B'$ the codomain and domain of $z^k$ respectively to differentiate. We now form the following fiber square:

\[\begin{tikzcd}
X' \arrow[r,"b_k"] \arrow{d} & X\arrow{d}\\
B' \arrow[r,"z^k"]& B
\end{tikzcd}\]

Observe that $X' \rightarrow X$ is a log \'etale morphism between log schemes. Moreover, the induced morphism $\Sigma(X') \rightarrow \Sigma(X)$ is cone wise given by a lattice refinement. In particular, $X' \rightarrow X$ factors through the root stack $X' \rightarrow X_k$ induced by the previous lattice refinement.  We wish to understand the log invariants of $X'$ in terms of the log invariants of $X$. To do so, suppose we have a realizable tropical type $\tau'$ of log stable map to $X'$. This induces a realizable tropical type $\tau$ of log stable map to $X$, and a morphism $st: \mathscr{M}(X',\tau') \rightarrow \mathscr{M}(X,\tau)$. Moreover, by identifying $\Sigma(X')$ with $\Sigma(X_k)$, $\tau'$ also defines a moduli stack $\mathscr{M}(X_k,\tau')$, and we have a morphism $i:\mathscr{M}(X',\tau') \rightarrow \mathscr{M}(X_k,\tau')$. Analogous morphisms also exist between the relative moduli stacks over the base $B$. We now state a corollary of the results of this section:

\begin{corollary}
With notation as above and $r = |coker(\tau^{'gp}_\NN \rightarrow \tau^{gp}_\NN)|$, we have $r \mid k$ and:

\begin{align*}
st_*([\mathscr{M}(X'/B',\tau')]^{vir}) &= \frac{k}{r}[\mathscr{M}(X/B,\tau)]^{vir}.\\
st_*([\mathscr{M}(X'/B,\tau')]^{vir}) &= \frac{1}{r}[\mathscr{M}(X/B,\tau)]^{vir}.\\
st_*([\mathscr{M}(X',\tau')]^{vir}) &= \frac{k}{r}[\mathscr{M}(X,\tau)]^{vir}.
\end{align*}

\end{corollary}

\begin{proof}
%

%
We first prove the claim $r \mid k$. Suppose we have an integral tropical map $h: \Gamma \rightarrow \Sigma(X)$ of type $\tau$, and let $v \in V(G_\tau)$ be a vertex. Then if $(h,\Gamma) \in \tau_\NN$ is in the image of $\tau'_\NN \rightarrow \tau_\NN$, we must have $\Sigma(\pi)h(v) \in k\NN \subset \NN$, with $\Sigma(\pi): \Sigma(X) \rightarrow \mathbb{R}_{\ge 0}$ the tropicalization of the map $\pi: X \rightarrow B$. On the other hand, if $\Sigma(\pi)h(v) \in k\NN$, then all other integral points of $\Gamma$ must also have image under $\Sigma(\pi)h$ contained in $k\NN$ since $G_\tau$ is connected and all contact orders of bounded edges of $G_\tau$ have images under $\pi$ given by integral tangent vectors in $k\ZZ$. Therefore, $(h,\Gamma) \in im(\tau'_\NN)$ if and only if $\Sigma(\pi)h(v) \in k\NN$. Hence, letting $h_v: \tau_\NN \rightarrow \Sigma(X)$ be the evaluation map at the vertex $v \in V(G_\tau)$, the quotient group $\tau_\NN^{gp}/\tau^{'gp}_\NN$ is isomorphic to the quotient $im(\Sigma(\pi)h_v)/(k\NN\cap im(\Sigma(\pi)h_v)$). Since both groups are cyclic groups, the quotient is cyclic of order dividing $k$.

We now show the equation $k[\mathscr{M}(X_k/B,\tau')]^{vir} = i_*([\mathscr{M}(X'/B',\tau')])^{vir}$. To do so, we will apply the Costello-Herr-Wise pushforward theorem of \cite{Cos} and \cite{pusherr}, in particular we need a cartesian squares relating the moduli stacks above. Let $\mathcal{X}' := \mathcal{A}_{X'}\times_{\mathcal{A}_\NN}B'$, and consider the commuting diagram:
\[\begin{tikzcd}
\mathscr{M}(X'/B',\tau') \arrow{r}\arrow{d} & \mathscr{M}(X_k/B,\tau')\arrow{d}\\
\mathfrak{M}(\mathcal{X}'/B',\tau') \arrow{r}& \mathfrak{M}(\mathcal{X}_k/B,\tau')
\end{tikzcd}\]

The claim that the above diagram is cartesian follows closely from the fact that $X'$ is defined via the following pullback diagram in all log categories:

\[\begin{tikzcd}
X' \arrow{r}\arrow{d} & X_k\arrow{d}\\
\mathbb{A}^1 \arrow{r} & \mathbb{A}^1/\mu_k
\end{tikzcd}\]

In particular, the data of a log stable map $C \rightarrow X'$ is the data of a log stable map $X_k$, together with a choice of $k^{th}$ root of the regular function $C \rightarrow \mathbb{A}^1$ induced by $C \rightarrow X_k$. Note in particular that $X' \rightarrow X_k$ is a strict log morphism. Moreover, since $B' \rightarrow B$ is log \'etale, we have $T_{X'/B'}^{log} = q^*T_{X_k/B}^{log}$ for $q: X' \rightarrow X_k$, hence the obstruction theory for the right vertical morphism pulls back to the obstruction theory for the left vertical morphism.

Consider the open substacks $\mathfrak{M}^\circ(\mathcal{X}_k/B,\tau')$ and $\mathfrak{M}^\circ(\mathcal{X}'/B',\tau')$ consisting of prestable log maps of type $\tau'$, and let $S \rightarrow \mathfrak{M}^\circ(\mathcal{X}_k/B,\tau')$ be a morphism. We claim that $S$ is in the image of $\mathfrak{M}^\circ(\mathcal{X}'/B',\tau')$. Indeed, the family of tropical maps $\Gamma \rightarrow \Sigma(X)$ associated to the morphism $\tau' \rightarrow \tau$ consists of tropical curves with edges and legs having integral lattices that already factor through the sublattice of interest. Hence, the resulting map $C/S \rightarrow \mathcal{X}$ is the middle horizontal morphism in Diagram \ref{maproot}. Since $\mathfrak{M}^\circ(\mathcal{X}'/B',\tau') \rightarrow \mathfrak{M}^\circ(\mathcal{X}_k/B,\tau')$ is degree $k$ onto its image, coming from the choice of $k^{th}$ root of the resulting map $C \rightarrow B$ induced by $C \rightarrow \mathcal{X}$, we find $\mathfrak{M}(\mathcal{X}'/B',\tau') \rightarrow \mathfrak{M}(\mathcal{X}_k/B,\tau')$ is degree $k$.

 In order to apply the Costello-Herr-Wise pushforward theorem, we must also show that the morphism is proper. First, by replacing the relevant moduli stacks with open substacks which contain the images of $\mathscr{M}(X'/B',\tau')$ and $\mathscr{M}(X_k/B,\tau')$ where necessary, we may assume the morphism is finite type. Hence, we may prove properness by proving the morphism satisfies the valuative criterion. To this end, suppose we have a valuation ring $R$ with field of fractions $K$, and morphisms $Spec\text{ }K \rightarrow \mathfrak{M}(\mathcal{X}'/B',\tau')$ and $Spec\text{ }R \rightarrow \mathfrak{M}(\mathcal{X}_k/B,\tau')$ making the obvious square commute via some chosen isomorphism of $K$-valued points of $\mathfrak{M}(\mathcal{X}_k/B,\tau')$. Note that contained in the data of the latter morphism is a prestable log curve $C_R$ over $Spec\text{ }R$, together with a non-basic map $f_R: C_R \rightarrow \mathcal{X}$ of type $\tau$. In particular, we have a regular map $Spec\text{ }R \rightarrow \mathbb{A}^1$ which pulls back to the map $C_R \rightarrow \mathbb{A}^1$ induced by $f_R$. 
 
 Letting $r \in R$ be the element corresponding to the former regular map, since $\mathcal{X}'$ is constructed as a fiber product in any log category, in order to lift $f_R$ to a morphism to $\mathcal{X}'$, it suffices to pick a $k^{th}$ root of $r$ in $R$. For existence, we note that $r$ has a $k^{th}$ root in the field $K$, and since valuations rings are integrally closed, $r$ has a $k^{th}$ root in $R$. Moreover, the morphism $Spec\text{ }K \rightarrow \mathfrak{M}(\mathcal{X}'/B',\tau')$ furnishes one such root. Thus, by the universal property defining $\mathcal{X}'$, there exists a pair of a map $C_R \rightarrow \mathcal{X}'$ and an isomorphism between the two resulting maps $C_R \rightarrow \mathcal{X}$ which is unique up to unique isomorphism. Moreover, the isomorphism realizing the lift is the unique such isomorphism which upon restriction to the dense open $Spec\text{ }C_K \rightarrow \mathcal{X}$ gives the isomorphism realizing the original commuting square. It follows from the valuative criterion that $\mathfrak{M}(\mathcal{X}'/B',\tau') \rightarrow \mathfrak{M}(\mathcal{X}_k/B,\tau')$ is proper. Hence,  the Costello-Herr-Wise pushforward theorem of \cite{Cos} and \cite{pusherr} allows us to conclude the desired equality of virtual classes. By Corollary \ref{birinvr}, we derive the first claimed equality of the corollary.

 
 For the second equality, we observe we have the following cartesian diagram:
 
 \[\begin{tikzcd}
 \mathscr{M}(X'/B',\tau') \arrow[r,"p"]\arrow{d} & \mathscr{M}(X'/B,\tau')\arrow{d}\\
 \mathfrak{M}(\mathcal{X}'/B',\tau') \arrow{r} & \mathfrak{M}(\mathcal{A}_{X'} \times_{\mathcal{A}_\NN} B/B,\tau')
 \end{tikzcd}\]
 As before, the both vertical morphisms are equipped with obstruction theories which are related by pullback. The bottom morphism is finite of degree $k$, hence we have $p_*([\mathscr{M}(X'/B',\tau')]^{vir}) = k[\mathscr{M}(X'/B,\tau')]^{vir}$ and the second claimed equality of the corollary follows.
 
 For the final equality, we cannot assume $\tau'$ defines a tropical type of map to $X'$ relative to the projection $X' \rightarrow B'$. We nonetheless have maps $\mathscr{M}(X',\tau')\rightarrow B'$ and $\mathscr{M}(X_k,\tau') \rightarrow B$, since every regular function on a family of proper curves pulls back from the base. Letting $\mathcal{X}_k$ be the Artin fan of $X_k$ relative to $Spec\text{ }\kk$, have the following cartesian diagram:
 
 \[\begin{tikzcd}
 \mathscr{M}(X',\tau') \arrow[r,"i"] \arrow{d} & \mathscr{M}(X_k,\tau')\arrow{d}\\
 \mathfrak{M}(\mathcal{A}_{X'},\tau')\times B' \arrow{r} & \mathfrak{M}(\mathcal{X}_k,\tau') \times B
 \end{tikzcd}\]
 
Again, the vertical morphisms are equipped with obstruction theories which are related by pullback. By similar arguments to proving the first claimed equality of the corollary, the morphism $\mathfrak{M}(\mathcal{A}_{X'},\tau') \rightarrow \mathfrak{M}(\mathcal{X}_k,\tau')$ is pure degree $1$. From this, we deduce the bottom horizontal of the above cartesian diagram is pure degree $k$, and therefore $i_*([\mathscr{M}(X',\tau')]^{vir}) = k[\mathscr{M}(X_k,\tau')]^{vir}$. By an application of Corollary \ref{birinvr}, we derive the third claimed equality of the corollary.


\end{proof}

%

\section{Comparing Invariants}
To begin the comparison, we note that given a log smooth pair $(X,D)$, there exists a log \'etale modification $(\tilde{X},\tilde{D}) \rightarrow (X,D)$ such that $\tilde{X}$ is smooth and $\tilde{D} \subset \tilde{X}$ is a simple normal crossings divisor. Moreover, after noting that the proof of Corollary $1.6$ of \cite{pbirinv} does not use the fact that the invariants in question define an associative algebra, if the structure constants defining $R_{(\tilde{X},\tilde{D})}$ satisfy the relations stated in Corollary \ref{mcr1}, then so to do the log Gromov-Witten invariants giving the structure constants for $R^{log}_{(X,D)}$. Moreover, by birational invariance of ordinary log Gromov-Witten invariants i.e. Corollary 1.3.1 of \cite{bir_GW}, in order to prove Theorem \ref{mthm2}, it suffices to prove the theorem for a log \'etale modification. Thus, without loss of generality, we will assume throughout that $X$ is smooth and $D$ has simple normal crossings. 

In order to prove Theorem \ref{mthm1}, the most important case for us to consider will be when the numerical data defining a type $\beta$ used in defining a given structure constant is slope sensitive, with output contact order $r$ corresponds to a point on a $1$-dimensional cone of $\Sigma(X)$ i.e., the contact order corresponds to a divisorial valuation on $\kk(X)$. In fact, we will consider the target $X_r$, which is a root stack of $X$ induced by a lattice refinement of $\Sigma(X_r)$ of $ \Sigma(X)$ such that $r$ is a primitive lattice point in $\Sigma(X_r)$. Let $\beta$ be a tropical type of punctured log curve, with curve class $\textbf{A} \in NE(X)$ and $l$ input contact orders $p_1,\ldots,p_l$, and an output contact order $-r$ with $r \in B(\ZZ)$ associated with a marked point $x_{out}$. By the assumption that $\beta$ is a slope sensitive and Remark \ref{idloget}, there is a corresponding tropical lift $\beta'$ of $\beta$ to $X_r$ and an analogous moduli stacks $\mathfrak{M}^{ev}(\mathcal{X}_r,\beta')$. 


We now note that we may define the evaluation space $\mathscr{P}(X_r,r)$ analogously to $\mathscr{P}(X,r)$. Letting $\mathscr{D}_r$ be the stratum of $X_r$ associated with $r$, the underlying stack of $\mathscr{P}(X_r,r)$ is $\mathscr{D}_r\times B\mathbb{G}_m$, and the log structure is the sub log structure of the natural product log structure on the product defined as in Display \ref{logev}. Finally, after noting that the image of the section of the universal twisted curve over $\mathscr{M}(X_r,\beta)$ has trivial inertia, the universal twisted curve has an underlying stack theoretic section with fiberwise image the marked point above, and the proof of Proposition $3.3$ of \cite{int_mirror} applies to produce evaluation maps $ev_{x_{out}}: \mathscr{M}(X_r,\beta) \rightarrow \mathscr{P}(X_r,r)$ and $ev_{x_{out}}: \mathfrak{M}^{ev}(\mathcal{X}_r,\beta) \rightarrow \mathscr{P}(X_r,r)$.

For any $z \in D_r$, the proof of Proposition $3.8$ of \cite{int_mirror} can also be used to produce a morphism $z': W \rightarrow \mathscr{P}(X_r,r)$ which when composed with the morphism $\mathscr{P}(X_r,r) \rightarrow \mathscr{P}(X,r)$ maps to $z$. Moreover, since $r$ is a primitive point along a $1$-dimensional cone of $\Sigma(X_r)$, the morphism $z'$ is strict. In particular, the fiber products $\mathscr{M}(X_r,\beta',z) := \mathscr{M}(X_r,\beta') \times_{\mathscr{P}(X_r,r)} W$ and $\mathfrak{M}^{ev}(X_r,\beta',z) := \mathfrak{M}^{ev}(\mathcal{X}_r,\beta')\times_{\mathscr{P}(X_r,r)} W$ are independent of the choice of log category in which they are taken. For a realizable type $\tau$ marked by $\beta$ we also consider the analogous fiber products $\mathscr{M}(X,\tau,z)$ and $\mathfrak{M}^{ev}(\mathcal{X}_\tau,z)$. The morphism $\mathscr{M}(X_r,\beta',z) \rightarrow \mathfrak{M}^{ev}(\mathcal{X}_r,\beta',z)$ can be equipped with a perfect obstruction theory as in \cite{int_mirror} with respect to which $\mathscr{M}(X_r,\beta',z)$ virtually is pure dimension $0$. We define an invariant analogous to those that define the structure constants of $R_{(X,D)}$ by:

\[N_{p,q,r}^{\textbf{A}'} = deg [\mathscr{M}(X_r,\beta',z)]^{vir}\]

For $\pmb\tau$ a realizable decorated tropical type marked by $\beta$, we also consider the analogous moduli stack $\mathscr{M}(X,\pmb\tau,z) := \mathscr{M}(X,\pmb\tau) \times_{\mathscr{P}(X,r)} W$, which is equipped with a virtual fundamental class $[\mathscr{M}(X,\pmb\tau,z)]^{vir}$. We can express this invariant in the display above as an integral on a moduli stack without a point constraints:

\begin{lemma}\label{rpointconst}
Letting $r = km$ with $m$ a primitive integral point of $\Sigma(X)$ and $k \ge 1$, for $\beta'$ a tropical lift of a tropical type $\beta$ which is slope sensitive, we have the equalities:

\[[\mathscr{M}(X_r,\beta',z)]^{vir} = [\mathscr{M}(X_r,\beta')]^{vir}\cap ev_{x_{out}}^*(pt)\]
In the above equation, the morphism $ev_{x_{out}}$ is the underlying stack evaluation maps to the divisor $\mathscr{D}_r \subset X_r$. Moreover, for any $\alpha \in A^*(\mathscr{M}(X,\beta,z))$, the vanishing of the classes $\alpha\cap[\mathscr{M}(X,\pmb\tau,z)]^{vir}$ for all decorated tropical types $\pmb\tau$ marked by $\beta$ implies $[\mathscr{M}(X_r,\beta')]^{vir}\cap (ev^*_{x_{out}}(pt)\cup \alpha) = 0$.

\end{lemma}

\begin{proof}
Throughout, we abuse notation and always let $z$ refer to the morphism from the point constrained moduli space to the unconstrained moduli space. Recall that the underlying stack of the point constrained moduli stacks are ordinary fiber products. As a result, the claimed formula now follows if we show the following equality:
\begin{equation}\label{pcasint}
z_*([\mathscr{M}(X_r,\beta')\times_{X_r}z]^{vir}) = [\mathscr{M}(X_r,\beta')]^{vir}\cap ev_{x_{out}}^*(pt).
\end{equation}
 To see this claim, we let $f: \mathscr{M}(X_r,\beta',z) \rightarrow \mathfrak{M}^{ev}(\mathcal{X}_r,\beta',z)$, and note that we have the following triple of morphisms which comes equipped with a compatible triple of obstruction theories:

\[\mathscr{M}(X_r,\beta',z) \rightarrow \mathscr{M}(X_r,\beta') \rightarrow \mathfrak{M}^{ev}(\mathcal{X}_r,\beta').\]

The obstruction theory for the left most arrow is pulled back from that of $\mathfrak{M}^{ev}(\mathcal{X}_r,\beta',z) \rightarrow \mathfrak{M}^{ev}(\mathcal{X}_r,\beta')$, which itself is pulled back from the natural obstruction theory for $z: Spec\text{ }\kk \rightarrow \mathscr{D}_r$. In particular, by Theorem $4.9$ of \cite{vpull}, we have:

\[[\mathscr{M}(X_r,\beta')\times_{X_r}z]^{vir} = f^!([\mathfrak{M}^{ev}(\mathcal{X}_r,\beta',z)]) = f^!(z^*([\mathfrak{M}^{ev}(\mathcal{X}_r,\beta')])) = z^*([\mathscr{M}(X_r,\beta')]^{vir}).\]
By pushing forward the equality along $z$, and the equality $z_*z^*(A) = A \cap ev_{x_{out}}^*(pt)$ for any $A \in A_*(\mathscr{M}(X,\beta'))$, we conclude Equation \ref{pcasint}.

To see the final claim, note that by Corollary \ref{birinvr} we have 
\[ip_*([\mathscr{M}(X_r,\pmb\tau')]^{vir})\cap ev_{x_{out}}(pt) = [\mathscr{M}(X,\pmb\tau)]^{vir}\cap ev_{x_{out}}(pt),\]
for some $i > 0$. Moreover, letting $s: \mathfrak{M}^{ev}(\mathcal{X},\tau,z) \rightarrow \mathfrak{M}^{ev}(\mathcal{X},\tau) \times_{\mathscr{P}(X,r)} W$ be the map between the fs and coherent fiber products, then since $W \rightarrow \mathscr{P}(X,r)$ is an integral morphism, $s$ is finite and surjective. Hence, there exists $l> 0$ such that: 
\[z_*s_*[\mathfrak{M}^{ev}(\mathcal{X},\tau,z)] = l[\mathfrak{M}^{ev}(\mathcal{X},\tau)]\cap ev_{x_{out}}^*(pt) = ilp_*([\mathfrak{M}^{ev}(\mathcal{X}_r,\tau')])\cap ev^*_{x_{out}}(pt)\]
In particular, by taking virtual pullback of the above equation and capping with the class $\alpha$, we find $\alpha\cap[\mathscr{M}(X,\pmb\tau,z)]^{vir} = 0$ only if $[\mathscr{M}(X_r,\tau')]^{vir}\cap (ev_{x_{out}}^*(pt)\cup \alpha) = 0$. Finally, by Corollary \ref{vanishing}, which also holds in the setting of punctured maps to a root stack, the latter vanishing for all types $\tau'$ implies [$\mathscr{M}(X_r,\beta')]^{vir} \cap (ev_{x_{out}}^*(pt)\cup \alpha) = 0$, as required. 
\end{proof}

We now remark that the log \'etale invariance of the numbers $N_{p,q,r}^{\textbf{A}}$ demonstrated in the proof \cite{pbirinv} Corollary $1.6$ generalizes to the context of root stacks:

\begin{proposition}\label{altstconst}
Letting $\beta'$ be a slope sensitive tropical type which pushes forward to the tropical type $\beta$ associated to a structure constant $N_{p,q,r}^{\textbf{A}}$, then we have the equality:

\[N_{p,q,r}^{\textbf{A}} = N_{p,q,r}^{\textbf{A}'}\]

\end{proposition}

\begin{proof}

In light of Proposition \ref{repfun}, which is the analogue of Theorem $1.2$ \cite{pbirinv} in the setting of a root stack modification, the claim follows essentially from the proof of Corollary $1.6$ of \cite{pbirinv}. More precisely, letting $\mathfrak{M}^{ev}(\mathcal{X},\tau,z) = \mathfrak{M}^{ev}(\mathcal{X},\tau)\times_{\mathscr{P}(X,r)} W$ and $i_{\tau}: \mathfrak{M}^{ev}(\mathcal{X},\tau,z) \rightarrow \mathfrak{M}^{ev}(\mathcal{X},\beta,z)$, recall from Section $10$ of \cite{pbirinv} that we have a decomposition:
\[[\mathfrak{M}^{ev}(\mathcal{X},\beta,z)] = \sum \frac{m_\tau}{|Aut(\tau)|} i_{\tau*}([\mathfrak{M}(\mathcal{X},\tau,z)_{red}]).\]
The multiplicity $m_\tau$ above is $|coker(ev_{\tau}^{*,gp}: \ZZ \rightarrow \ZZ)|$, where $\tau$ is a minimal tropical type marked by $\beta$, necessarily satisfying $dim\text{ }\tau = 1$, and $ev_{\tau}: \tau_\NN\times_{\pmb\sigma(v)_\NN} \NN \rightarrow \NN$ is the evaluation morphism at the vertex $v$ containing the output leg of slope $-r$. By taking the virtual pullback and degree of the above equation, we have a decomposition $N_{p,q,r}^\textbf{A} = \sum_{\pmb\tau} N_{\pmb\tau}$, with $N_{\pmb\tau}$ defined as the degree of the contribution to $[\mathscr{M}(X,\beta,z)]^{vir}$ from decorated type $\pmb\tau$. 

To produce the analogous decomposition of $N_{p,q,r}^{\textbf{A}'}$, the main ingredient missing in order to apply the argument given in Section $10$ of \cite{pbirinv} is $\mathfrak{M}^{ev}(\mathcal{X}_r,\beta',z) \rightarrow W$ being log smooth. After recalling from remark \ref{idloget} that $\mathfrak{M}(\mathcal{X}_r,\beta') \rightarrow \mathfrak{M}^{tw}$ is idealized log \'etale, and that $\mathfrak{M}^{tw}$ is log smooth as follows from \cite{ltcur} Theorem $1.10$, slight modifications made to the proofs of Propositions $3.14$, Theorem $3.16$, and Lemma $3.17$ of \cite{int_mirror} show that $\mathfrak{M}^{ev}(\mathcal{X}_r,\beta',z) \rightarrow W$ is log smooth.


Using the stratification of $\mathfrak{M}(\mathcal{X}_r,\beta')_{red}$ by the stacks $\mathfrak{M}(\mathcal{X}_r,\tau')$ remarked upon in Remark \ref{idloget}, the argument from Section $10$ of \cite{pbirinv} now applies to produce the analogous decompositions:

\[[\mathfrak{M}^{ev}(\mathcal{X}_r,\beta',z)] = \sum_{\pmb\tau'} \frac{m_{\tau'}}{|Aut(\tau')|} i_{\tau^{'*}}([\mathfrak{M}(\mathcal{X}_r,\tau',z)_{red}]).\]

Virtual pullback of the above equation to $A_*(\mathscr{M}(X,\beta',z'))$ produces a decomposition $N_{p,q,r}^{\textbf{A}'} = \sum_{\pmb{\tau}'} N_{\pmb\tau'}$ as before. The main claim now reduces to showing $N_{\pmb\tau} = N_{\pmb\tau'}$, which is the analogue of Proposition $10.1$ of \cite{pbirinv} in the root stack setting. Using the fact  that the log point constraint $z: W \rightarrow \mathscr{P}(X,r)$ uniquely determines a point constraint $W \rightarrow \mathscr{P}(X_r,r)$ together with the the expression of $\mathscr{M}(X_r,\pmb\tau')$ as a fiber product implied by Proposition \ref{repfun}, the proof of Proposition $10.1$ of loc. cit. applies in the present context to prove the desired equality. 

\end{proof}

\begin{corollary}\label{pointcon}
Letting $z: \mathscr{M}(X,\beta,z) \rightarrow \mathscr{M}(X,\beta)$ be the projection from the fiber product and $ev: \mathscr{M}(X,\beta) \rightarrow D_r$ the evaluation map, we have:
\[z_*([\mathscr{M}(X,\beta,z)]^{vir}) = kp_*([\mathscr{M}(X_r,\beta')]^{vir})\cap ev^*(pt).\]

\end{corollary}

\begin{proof}
Note that by the projection formula, we may compute the right hand side of the desired equation by pulling back the class $ev^*(pt)$ along the morphism $p$ and capping the resulting class with $[\mathscr{M}(X_r,\beta')]^{vir}$. This pullback can also be expressed by pulling back the class of a point along $\mathscr{D}_r \rightarrow D_r$ followed by pulling back along the evaluation map. The resulting class is $\frac{1}{k}ev^*([pt'])$, where $[pt']\in A^{dim\text{ }X}(X_r)$ is the class associated with a closed immersion $B\mu_r \rightarrow X_r$. Since the left hand side of the desired equation is equal to the integral on $\mathscr{M}(X_r,\beta')$ with insertion $ev^*([pt])$ by Proposition \ref{altstconst}, the desired equation follows.
\end{proof}

\begin{corollary}\label{logorb}
Let $\beta$ be a tropical type introduced at the start of this section with $r = km$ for $m\in B(\ZZ)$ a primitive integral point in a $1$-dimensional cone of $\Sigma(X)$ and total curve class $\textbf{A} \in NE(X)$, and suppose $X$ is smooth and slope sensitive with respect to $\mathscr{M}(X,\beta)$ in the sense of \cite{logroot}. Then for any $a \ge 0$, $\psi_{x_{out}}^a[\mathscr{M}(X,\pmb\tau,z)]^{vir} = 0$ for all decorated realizable types $\pmb\tau$ marked by $\beta$ only if $\langle\prod_i\psi_{x_i}^{a_i}[\alpha_i],\psi_{x_{out}}^a[pt]\rangle_{\textbf{A}} = 0$ for all choices of $a_i\ge 0$ and insertions $\alpha_i \in H^*(D_{p_i})$. Moreover, we have $N_{p_1,p_2,r}^\textbf{A} = N_{p_1,p_2,r}^{orb\textbf{A}}$.
\end{corollary}

\begin{proof}
Let $D = D_1 +\cdots + D_p$ be a decomposition of $D$ into irreducible components with $D_1 = D_r$. Note by the smoothness assumption that $D_i$ is a Cartier divisor for all $i$. For $\textbf{b} = (b_1,\ldots,b_p)$, let $X_{\textbf{b}}$ denote the root stack of $X$ given by taking the $b_i^{th}$ root stack of the divisor $D_i$ for all $i$. Note that by the notation introduced in display \ref{orbinv}, for $b_i \gg 0$, we have:

\begin{equation}
\langle\prod_i\psi_{x_i}^{a_i}[\alpha_i],\psi_{x_{out}}^a[pt]\rangle_{\textbf{A}}  = b_1^{m_{i,-}}\prod_{i> 1}b_i^{m_{i,-}}\int_{[\mathscr{M}^{orb}(X_{\textbf{b}},\beta)]^{vir}} \psi_{x_{out}}^a ev_{x_{out}}^*([pt])\prod_{i=1}^l \psi_{x_i}^{a_i}ev_{x_i}^*([\alpha_i]).
\end{equation}
Alternatively, since multiples of $k$ are cofinal in the poset of the natural numbers with the divisibility relation, it suffices to consider the inverse system over root stacks of the pair $(X,D)$ with a rooting parameter for $D_1$ divisible by $k$. Hence, letting $X_r$ be the log \'etale modification of $X$ in which we take $k^{th}$ root of the the Cartier divisor $D_1$, we then have for $b_i\gg 0$:

\begin{equation}\label{altorb}
\langle\prod_i\psi_{x_i}^{a_i}[\alpha_i],\psi_{x_{out}}^a[pt]\rangle_{\textbf{A}} = (kb_1)^{m_{1,-}}\prod_{i>1}b_i^{m_{i,-}}\int_{[\mathscr{M}^{orb}((X_r)_{\textbf{b}},\beta)]^{vir}} \psi_{x_{out}}^a ev_{x_{out}}^*([pt])\prod_{i=1}^l \psi_{x_i}^{a_i}\alpha_i.
\end{equation}

On the other hand, using \cite{logroot2} Theorem $B$ for the first equality, which holds after replacing $\mathscr{M}(X,\beta)$ with $\mathscr{M}(X_r,\beta)$, and Equation \ref{altorb} for the second equality, we have:
\[
\begin{split}\label{lrorb}
\int_{[\mathscr{M}(X_r,\beta)]^{vir}} \psi_{x_{out}}^a ev_{x_{out}}^*([pt])\prod_{i=1}^l \psi_{x_i}^{a_i}\alpha_i &= b_1^{m_{1,-}}\prod_{i>1}b_i^{m_{i,-}}\int_{\small{[\mathscr{M}^{orb}((X_r)_\textbf{b},\beta)]^{vir}}} \psi_{x_{out}}^a ev_{x_{out}}^*([pt])\prod_{i=1}^p \psi_{x_i}^{a_i}\alpha_i\\
&= \frac{1}{k^{m_{1,-}}}\langle\prod_{i>1}\psi_{x_i}^{a_i}[\alpha_i],\psi_{x_{out}}^a[pt]\rangle_{\textbf{A}} .
\end{split}
\]
By the second part of Lemma \ref{rpointconst}, the vanishing hypothesis gives the left hand side of the above expression is zero for all $\alpha_i$, hence the righthand side vanishes, as required for the first claim of the corollary. The claimed equality of structure constants follows by multiplying both sides of Equation \ref{lrorb} by $k$, and combining the resulting equality with the equality deduced in Corollary \ref{pointcon}.



\end{proof}

Using the corollary above, Theorem \ref{mthm1} follows in a straightforward manner from the birational invariance of the invariants $N_{p,q,r}^{\textbf{B}}$ proven in \cite{pbirinv} Corollary $1.6$:

\begin{proof}[Proof of Theorem \ref{mthm1}]
Letting $p,q,r \in B(\ZZ)$ and $\textbf{B} \in NE(X)$, after packing this discrete data into the type $\beta$, consider the moduli space $\mathscr{M}(X,\beta)$ of punctured log maps. Pick a slope sensitive modification $\pi: \tilde{X} \rightarrow X$ with respect to $\beta$, i.e. a log \'etale modification such that any lift of the discrete data $\beta$ to $\tilde{X}$ is slope sensitive. After potentially further blowups, we may assume  $r$ is contained in a $1$-dimensional cone of $\Sigma(\tilde{X})$. Thus, the requirements of Corollary \ref{logorb} are satisfied, and for any $\textbf{A} \in NE(\tilde{X})$, and we have the equality:

\begin{equation}\label{eq1}
N_{p_1,p_2,r}^\textbf{A} = N_{p_1,p_2,r}^{orb\textbf{A}}
\end{equation}
Moreover, by Corollary $1.6$ of \cite{pbirinv}, we have:
\begin{equation}\label{eq2}
N_{p_1,p_2,r}^{\textbf{B}} = \sum_{\pi_*(\textbf{A}) = \textbf{B}} N_{p_1,p_2,r}^{\textbf{A}}
\end{equation}
Substituting the right hand side of Equation \ref{eq1} for terms on the right hand side of \ref{eq2} gives the desired result. 
\end{proof}

We now assume the pair $(X,D)$ satisfies condition $1$ of Assumption \ref{lcy}. Under this assumption, Tseng and You proved that the invariants above give the structure constants of an algebra over $S_I$. Using Theorem \ref{mthm1}, together with birational invariance of the structure constants used in the definition of the logarithmic mirror algebra, we give a proof of associativity of the log mirror algebra with the same assumptions.


\begin{proof}[Proof of Corollary \ref{mcr1} with Assumtion \ref{lcy}(1)]

Let $p_1,p_2,p_3 \in B(\mathbb{Z})$, and consider the $z^{\textbf{A}}\vartheta_r$ term of the two following products of theta functions:

\begin{align}
\vartheta_{p_1}(\vartheta_{p_2}\vartheta_{p_3})[z^{\textbf{A}}\vartheta_r] &= \sum_{\textbf{A} = \textbf{A}_1 + \textbf{A}_2\text{, }s \in B(\mathbb{Z})} N_{p_1,s,r}^{\textbf{A}_1}N_{p_2,p_3,s}^{\textbf{A}_2}\\
(\vartheta_{p_1}\vartheta_{p_2})\vartheta_{p_3}[z^{\textbf{A}}\vartheta_r] &= \sum_{\textbf{A} = \textbf{A}_1 + \textbf{A}_2\text{, }s \in B(\mathbb{Z})} N_{p_1,p_2,s}^{\textbf{A}_1}N_{p_3,s,r}^{\textbf{A}_2}.
\end{align}

By Proposition \ref{balancing}, only finitely many choices of $s \in B(\ZZ)$ for a given choice of splitting $\textbf{A}_1+\textbf{A}_2 = \textbf{A}$ could possibly yield a non-zero contribution to right hand side of the equalities above. Hence, only finitely many moduli space of punctured log maps appear in defining this collection of structure constants. We may therefore pass to a sufficiently fine subdivision of $(X,D)$, $(\tilde{X},\tilde{D})$, such that each of the resulting types $\beta$ are slope sensitive. By birational invariance of the structure constants, we find the two products of $\vartheta$ functions above are respectively equal to:

\begin{equation}\label{liftprod}
\vartheta_{p_1}(\vartheta_{p_2}\vartheta_{p_3})[z^{\textbf{A}}\vartheta_r] = \sum_{\textbf{A} = \textbf{A}_1 + \textbf{A}_2\text{, }s \in B(\mathbb{Z})}\sum_{\pi_*(\textbf{B}_i) = \textbf{A}_i} N_{p_1,s,r}^{\textbf{B}_1}N_{p_2,p_3,s}^{\textbf{B}_2}
\end{equation}
\begin{equation}\label{liftprod2}
(\vartheta_{p_1}\vartheta_{p_2})\vartheta_{p_3}[z^{\textbf{A}}\vartheta_r] = \sum_{\textbf{A} = \textbf{A}_1 + \textbf{A}_2\text{, }s \in B(\mathbb{Z})}\sum_{\pi_*(\textbf{B}_i) = \textbf{A}_i} N_{p_1,p_2,s}^{\textbf{B}_1}N_{p_3,s,r}^{\textbf{B}_2}.
\end{equation}

 
By associativity of the orbifold mirror algebra i.e. Theorem $37$ of \cite{RQC_no_log}, after fixing $\textbf{B} \in NE(\tilde{X})$ such that $\pi_*(\textbf{B}) = \textbf{A}$, we have:
 
 \begin{equation}\label{assonbl}
 \sum_{\textbf{B} = \textbf{B}_1 + \textbf{B}_2\text{, }s \in B(\mathbb{Z})} N_{p_1,s,r}^{\textbf{B}_1,orb}N_{p_2,p_3,s}^{\textbf{B}_2,orb} = \sum_{\textbf{B} = \textbf{B}_1 + \textbf{B}_2\text{, }s \in B(\mathbb{Z})} N_{p_1,p_2,s}^{\textbf{B}_1,orb}N_{p_3,s,r}^{\textbf{B}_2,orb}.
 \end{equation}
 
As before, for a given splitting $\textbf{B} = \textbf{B}_1 + \textbf{B}_2$, there are finitely many choices of $s \in B(\ZZ)$ which could possibly yield a non-zero contribution to either side of the equality above. In fact, the moduli stack used to define the invariant $N_{p_1,p_2,s}^{\textbf{B}_1}$ satisfies the condition of Proposition \ref{balancing} required for non-emptiness only if the moduli stack used to define $N_{p_1,p_2,s}^{\pi_*(\textbf{B})}$ satisfies the condition of Proposition \ref{balancing}. Hence, both side of Equation \ref{assonbl} only include invariants associated with discrete data potentially contributing to either equation \ref{liftprod} or \ref{liftprod2}. In particular, by our choice of blowup and Corollary \ref{logorb}, we have $N_{p,q,r}^\textbf{B} = N_{p,q,r}^{\textbf{B},orb}$ for all invariants appearing in the relation of orbifold invariants above.

 By summing over all choices of curve classes $\textbf{B}$ with $\pi_*(\textbf{B}) = \textbf{A}$ and appropriately rearranging terms using the birational invariance of the structure constants, we derive the desired equality. More explicitly, by Corollary $1.6$ of \cite{pbirinv}, we have
\[N_{p_1p_2r}^\textbf{A} = \sum_{\pi_*(\textbf{B}) = \textbf{A}} N_{p_1p_2r}^{\textbf{B}}\]
Using this equation, we have the following equality:

\[
\begin{split}
\sum_{\pi_*(\textbf{B}) = \textbf{A}} \sum_{\textbf{B} = \textbf{B}_1 + \textbf{B}_2\text{, }s \in B(\mathbb{Z})} N_{p_1,s,r}^{\textbf{B}_1}N_{p_2,p_3,s}^{\textbf{B}_2}&= \sum_{\textbf{A} = \textbf{A}_1 + \textbf{A}_2}(\sum_{\pi_*(\textbf{B}_1) = \textbf{A}_1} N_{p_1,s,r}^{\textbf{B}_1})(\sum_{\pi_*(\textbf{B}_2)=\textbf{A}_2} N_{p_2,p_3,s}^{\textbf{B}_2}) \\
&= \sum_{\textbf{A} = \textbf{A}_1 + \textbf{A}_2} N_{p_1,s,r}^{\textbf{A}_1}N_{p_2,p_3,s}^{\textbf{A}_2} \\
&= \vartheta_{p_1}(\vartheta_{p_2}\vartheta_{p_3})[z^\textbf{A}\vartheta_r].
\end{split}
\]
Similarly, we have:

\[\sum_{\pi_*(\textbf{B}) = \textbf{A}} \sum_{\textbf{B} = \textbf{B}_1 + \textbf{B}_2\text{, }s \in B(\mathbb{Z})} N_{p_1,p_2,s}^{\textbf{B}_1}N_{p_3,s,r}^{\textbf{B}_2} = \sum_{\textbf{A} = \textbf{A}_1 + \textbf{A}_2} N_{p_1,p_2,s}^{\textbf{A}_1}N_{p_3,s,r}^{\textbf{A}_2} = (\vartheta_{p_1}\vartheta_{p_2})\vartheta_{p_3}[z^\textbf{A}\vartheta_r] .\]
Since the left hand sides of the two above equations are equal by Equation \ref{assonbl}, the right hand sides are equal, as required. 

\end{proof}

\section{WDVV, TRR, and intrinsic mirror symmetry for log Calabi-Yau varieties}

%
%
%
To prove our main theorems for all log Calabi-Yau pairs considered in \cite{int_mirror}, we recall the following two relations which hold in the large $r$ genus $0$ orbifold Gromov-Witten theory for any simple normal crossings pair $(X,D)$, see \cite{RQC_no_log} Propositions $27$ and $26$ respectively:

\begin{proposition}[WDVV relation]\label{WDVV}

Let $\textbf{A} \in NE(X)$, and a collection of $m$ pairs $(s_i,[\gamma_i]_{s_i})$ with $s_i \in \Sigma(X)(\ZZ)$ and $[\gamma_i]_{s_i} \in H^*(D_s)$. After picking a graded basis $\{T_{s,k}\}\subset H^*(D_s)$, with a dual basis $\{T_s^k\}\subset H^*(X_s)$ under the intersection pairing on $H^*(X_s)$, then by summing over all tuples $(\textbf{A}_1,\textbf{A}_2,s,S_1,S_2,T_{s,k})$ with $\textbf{A}_1+\textbf{A}_2 = \textbf{A}$, $S_1\sqcup S_2 = [m]\setminus[4]$, $s \in \ZZ^m$ and $T_{s,k}$ an element of the chosen basis $H^*(X_s)$.
\begin{equation}\label{WDVV1}
\begin{split}
\sum_{(\textbf{A}_1,\textbf{A}_2,S_1,S_2,s,T_{s,k})} \langle[\gamma_1]_{s_1},[\gamma_2]_{s_2},\prod_{j \in S_1}[\gamma_j]_{s_j},T_{-s,k}\rangle_{\textbf{A}_1}\langle T^k_{s},[\gamma_3]_{s_3},[\gamma_4]_{s_4},\prod_{j\in S_2} [\gamma_j]_{s_j}\rangle_{\textbf{A}_2} \\
= \sum_{(\textbf{A}_1,\textbf{A}_2,s,S_1,S_2,T_{s,k})} \langle[\gamma_2]_{p_2},[\gamma_3]_{p_3},\prod_{j \in S_1} [\gamma_j]_{s_j},T_{-s,k}\rangle_{\textbf{A}_1}\langle T^k_{s},[\gamma_1]_{p_1},[\gamma_4]_{s_4},\prod_{j \in S_2} [\gamma_j]_{s_j}\rangle_{\textbf{A}_2}.
\end{split}
\end{equation}
\end{proposition}

\begin{proposition}[Topological recursion relation]\label{TRR}
With notation as above but with $S_1\sqcup S_2 = [m]\setminus [3]$, and $a_1,\ldots,a_l \in \NN$, we have the following equality of large $r$ orbifold Gromov-Witten invariants:
\begin{equation}\label{TRR1}
\begin{split}
\langle\prod_i [\gamma_i]_{s_i}\psi^{a_i}\rangle= \sum_{(\textbf{A}_i,s,S_i,T_{s,k})} \langle \psi^{a_1-1}[\gamma_1]_{s_1},\prod_{j \in S_1} \psi^{a_j}[\gamma_j]_{s_j},T_{s,k}\rangle_{\textbf{A}_1}\langle T_{-s}^k, \psi^{a_2}[\gamma_2]_{s_2},\psi^{a_3}[\gamma_3]_{s_3},\prod_{j \in S_2} \psi^{a_j}[\gamma_j]_{s_j}\rangle_{\textbf{A}_2}.
\end{split}
\end{equation}

\end{proposition}

As before, in the logarithmic setting, we will have access to analogous relations only if we assume all the moduli spaces determined by the discrete data appearing in equations above are slope sensitive. More explicitly, after fixing the initial data of the above two propositions, let us say a decorated tropical type $\beta = (\textbf{A}_1,((s_i),s))$ is a \emph{potentially contributing type} if the type is balanced in the sense of Proposition \ref{balancing} and there exists a curve class $\textbf{A}_2$ such that $\textbf{A}_1 + \textbf{A}_2 = \textbf{A}$. By Proposition \ref{balancing}, the space of orbifold maps with discrete data $(\textbf{A}_1,((s_i),s))$ is non-empty only if the discrete data is potentially contributing. Moreover, there are only finitely many potentially contributing types. Indeed, by fixing an ample line bundle $L$ on $X$ and defining $deg(\textbf{B}) = \textbf{B}\cdot c_1(L)$, there are only finitely many curves $\textbf{A}'$ such that $deg(\textbf{A}') \le deg(\textbf{A})$. Moreover, by Proposition \ref{balancing}, there are only finitely many choices of contact order $s$ after having fixed $\textbf{A}_1$ and $s_i$. 

As a result, there are only finitely many non-zero possible terms in either of the above relations. Therefore, by Section $3.5$ of \cite{logroot}, there exists a smooth projective resolution $\tilde{X} \rightarrow X$ such that any lift of a potentially contributing tropical type to $\Sigma(\tilde{X})$ is slope-sensitive. In addition, it is straightforward to check that any potentially contributing type for $\tilde{X}$ pushes forward to a potentially contributing type for $X$. Hence, all potentially contributing types to the WDVV and TRR equations for $\tilde{X}$ are slope sensitive. In particular, we produce relations in the log Gromov-Witten theory for $\tilde{X}$ by the main result of \cite{logroot2}. We record this fact in the following theorem:

\begin{theorem}\label{logrelations}
Given a curve class $\textbf{A} \in NE(X$), and integral points $s_i \in \Sigma(X)(\ZZ)$ for $1\le i \le m$, there exists a smooth projective log \'etale modification $\pi: \tilde{X} \rightarrow X$ such that after a choice of lift $s_i' \in \Sigma(\tilde{X})(\ZZ)$, for any class $\textbf{B}$ with $\pi_*(\textbf{B}) = \textbf{A}$, equations \ref{WDVV1} and \ref{TRR1} hold after replacing all orbifold Gromov-Witten invariants with corresponding refined punctured log Gromov-Witten invariants defined in \cite{logroot2}. 
\end{theorem}

\begin{remark}
When the target pair is toric, the genus $0$ log GW theory is entirely captured by tropical geometry. As a result, when $X$ is a surface, one should recover results of Gathman, Markwig and Rau in \cite{GaMa}, \cite{tropdecGW}, and for general toric $X$ the WDVV and TRR relations investigated in \cite{tropint}. The latter paper including a hypothesis of $\Theta$ directionality, which appears to be implied by the assumption that all moduli spaces involved are slope sensitive. The source of complexity requiring this extra condition comes from issues with splitting the diagonal of the target, analogous to the problems found in gluing formulas for log GW invariants.  
\end{remark}

%

%
%

In order to use these formula, one needs to know a log \'etale modification of the target space in which the tropical types are slope sensitive, which typically requires a good understanding of the tropical types associated to strata of the relevant moduli space. Assuming this is known however, the extra cohomology classes appearing in the formula come correspond to piecewise polynomial functions on the tropicalization, see for instance \cite{logDR} Section $6.6.2$. In this setting however, the relevant classes will be fundamental classes of strata as well as their Poincar\'e dual point classes, leading to a universal formula that assumes no further geometric input aside from assumption of log Calabi-Yau.


For our application of interest, we will be considering on the one hand $m = 4$, with $p_1,p_2,p_3,-r \in \Sigma(X)(\ZZ)$ and insertions $[1]_{p_1},[1]_{p_2},[1]_{p_3},[pt]_{-r}$, and on the other hand $m= k+1$ with $p_0 = 0$, $p_i \in \Sigma(X)(\ZZ)$ for $i\not= 0$, and insertions $[pt]_{p_0}\psi^{k-2},[1]_{p_1},\ldots,[1]_{p_k}$.  In both cases of interest, equations \ref{WDVV1} and \ref{TRR1} do not yield the desired relations, but only after showing various terms contributing to summed expressions vanish. To deduce the necessary vanishing of orbifold invariants, we will show an associated log invariant vanishes, and use the log-orbifold correspondence of \cite{logroot} to prove the necessary vanishing of orbifold invariants. 

\begin{proof}[Proof of Corollary \ref{mcr1}]

Consider the WDVV relation described following Equation \ref{WDVV1}:

\begin{equation}\label{WDVV2}
\begin{split}
\sum_{(\textbf{A}_1, \textbf{A}_2,s,k)} \langle[1]_{p_1},[1]_{p_2},T_{-s,k}\rangle_{\textbf{A}_1}\langle T^k_{s},[1]_{p_3},[pt]_{-r}\rangle_{\textbf{A}_2} \\
= \sum_{(\textbf{A}_1, \textbf{A}_2 ,s,k)} \langle[1]_{p_2},[1]_{p_3},T_{-s,k}\rangle_{\textbf{A}_1}\langle T^k_{s},[1]_{p_1},[pt]_{-r}\rangle_{\textbf{A}_2}.
\end{split}
\end{equation}

As only finitely many orbifold invariants are involved in the above relation, we pass to a smooth projective log modification $\tilde{X} \rightarrow X$ in which all tropical lifts of all tropical types corresponding to strata of the associated log moduli problems are slope sensitive, and $r \in \Sigma(X)^{[1]}$. By Theorem \ref{logrelations}, we know all contributing combinatorial types to the orbifold WDVV relation for $\tilde{X}$ are slope sensitive.

We must show that the only possible terms which give a non-zero contribution to either side of equation \ref{WDVV2} for the target $\tilde{X}$ have insertions either $[1]_s$ or $[pt]_s$, with one of $\pm s$ a positive contact order. Since each is a product of two orbifold Gromov-Witten invariants, we will show that for terms with insertion not of this form, then at least one of them is zero. We prove this statement using both the log/orbifold correspondence established in \cite{logroot2}, and a modification of the proof Lemma $7.14$ of \cite{int_mirror}, referred to as the no-tail lemma in loc. cit.

\begin{lemma}\label{asslem}
If $\langle T_{s}^k,[1]_{p},[pt]_{-r}\rangle_{\textbf{A}_2} \not= 0$, then $s \in B(\ZZ)$ and $T_{s}^k = [1]_s$.
\end{lemma}

\begin{proof}
We bundle the degree $\textbf{A}_2$ and contact orders $-r,p,s$ into the type $\beta$, and consider the associated moduli space of punctured log maps $\mathscr{M}(X,\beta)$. By assumption, we have $r \in \Sigma^{[1]}(X)$. Thus, by Corollary \ref{logorb}, to show $\langle T_{s},[1]_{p},[pt]_{-r}\rangle_{\textbf{A}_2} = 0$, it suffices to consider the point constrained moduli spaces $\mathscr{M}(X,\tau,z)$ for all realizable tropical types $\tau$ marked by $\beta$, and show that the virtual classes $[\mathscr{M}(X,\tau,z)]^{vir}$ are all zero. We will show this is the case unless we have $s \in B(\ZZ)$.


%

To begin, note that by the point constraint, $G_\tau$ must contain a vertex $v_0$ containing the leg with contact order $-r$ mapping into the cone $\sigma_r$ with $r \in \sigma_{r,\NN}$. In particular, this vertex must map to a cone in $\mathscr{P}$. If $v_0$ were only $2$-valent, then we may cut along the unique edge containing $v_0$ to produce two tropical types which glue to give $\tau$, one of which consists of a single vertex with $2$ legs. By the vanishing tail lemma of \cite{int_mirror} i.e. Theorem $7.14$ of loc cit, we must have $[\mathscr{M}(X,\pmb\tau,z)]^{vir} = 0$. 

Thus, we suppose $v_0$ is at least $3$-valent. We define the spine of $\tau$ be the subgraph of $G_\tau$ given by the convex hull of the legs with contact orders $-r$ and $s$.  We will now show that each edge in this spine is contained in $B$:

\begin{lemma}\label{spine}
Let $\tau$ be a tropical type with legs $L,l_{out},l_2$, with $u_{l_{out}}= -r$ and $u_{l_2} = p$, and a vertex $v \in V(G_\tau)$ with $\pmb\sigma(v) \in \mathscr{P}(X)$ and contained in the legs $l_{out},l_2$. If $[\mathscr{M}(X,\pmb\tau,z)]^{vir} \not= 0$, then for all edges and legs $e \in E(G_\tau) \cup L(G_\tau)$ contained in the convex hull of $v$ and $L$, we have $\pmb\sigma(e) \subset B$. 

\end{lemma}

\begin{proof}
We prove the lemma by induction on length of the chain of edges and legs in the convex hull of $L$ and $l_{out}$. To that end, enumerate the edges in this convex hull by $l_{out} = e_0,e_1,\ldots, e_J = L$, with $e_1$ the unique edge in the convex hull which contains $v_0$ which is not equal to $L_{out}$ nor $l_2$. As a base case, note that $r \in B(\ZZ)$, and since $v_0 \in B$ by the point constraint, the leg with contact order $-r$ must be contained in $B$, i.e. $\pmb\sigma(l_{out}),\pmb\sigma(l_2) \subset B$. Now for the induction hypothesis, we suppose we have an edge $e_k$ in the spine such that $\pmb\sigma(e_{k-1}) \subset B$. In particular, $D_i^*(u_{e_{k-1}}) = 0$ for all bad divisors $D_i$. Letting $v \in V(G_\tau)$ be the shared vertex of $e_{k}$ and $e_{k-1}$, consider the two tropical types produced by cutting $\tau$ along $e_k$, and let $\tau_1$ the resulting tropical type containing $v$. By the induction hypothesis, we may assume $\pmb\sigma(e_{k-1}) \in B$, so in particular $v \in B$. Hence, $D_i^*(u_{e_k})\ge 0$ for any bad divisor $D_i$. Note now that since $c_1(T_X^{log}) = \sum_{i} -a_iD_i$, with $a_i \ge 0$ and $a_i = 0$ if and only if $\sigma_{D_i} \in B \subset \Sigma(X)$, and \cite{int_mirror} Corollary $1.14$, we have:
\[A(\pmb\tau_1) \cdot c_1(T_X^{log}) = \sum_{D_i\text{ }bad\text{ }divisor} -a_iD_i^*(u_{e_k}+p-r) = \sum_{D_i\text{ }bad\text{ }divisor} -a_iD_i^*(u_{e_k}) \le 0.\]
Hence we must have $A(\tau_1) \cdot c_1(T_X^{log}) \le 0$, with $A(\tau_1) \cdot c_1(T_X^{log}) = 0$ if and only if $D_i^*(u_{e_k})= 0$ for all bad divisors $D_i$. Since $v \in B$, this would also imply $e_{k} \in B$. Thus, suppose $A(\tau_1)\cdot c_1(T_X^{log}) < 0$. Letting $x_{out}$ and $x$ correspond to the special sections associated with the legs $l_{out}$ and $L$ respectively, note that by construction and Corollary $7.3$ of \cite{punc}, we have the following cartesian diagram, with the right vertical arrow equipped with the product of obstruction theories:
\[\begin{tikzcd}
\mathscr{M}(X,\tau,z) \arrow{r}\arrow{d}& \mathscr{M}(X,\tau_1,z)\times\mathscr{M}(X,\tau_2) \arrow{d}\arrow{r}& \mathscr{M}(X,\tau_1)\times \mathscr{M}(X,\tau_2)\arrow{d}\\
\mathfrak{M}^{ev(x_{out})}(\mathcal{X},\tau,z) \arrow{r} & \mathfrak{M}^{ev(x_{out},x)}(\mathcal{X},\tau_1,z)\times \mathfrak{M}^{ev(x)}(\mathcal{X},\tau_2) \arrow{r}& \mathfrak{M}^{ev(x_{out},x)}(\mathcal{X},\tau_1)\times \mathfrak{M}^{ev(x)}(\mathcal{X},\tau_2).
\end{tikzcd}\]
Moreover, the obstruction theories for the middle and left vertical arrow is the pullback of the obstruction theory of the right vertical arrow. By \cite{int_mirror} Theorem $A.13$, if we show $\mathscr{M}(X,\pmb\tau_1,z)$ has negative virtual dimension, we will have shown that $\pmb\tau$ cannot contribute.


With this end in mind, we note that the perfect obstruction theory for the morphism $\mathscr{M}(X,\pmb\tau_1,z) \rightarrow \mathfrak{M}^{ev(x_{out},x)}(\mathcal{X},\tau_1,z)$ gives $\mathscr{M}(X,\pmb\tau_1,z)$ the following relative virtual dimension over $\mathfrak{M}^{ev(x_{out},x)}(\mathcal{X},\tau_1,z)$:

\[\chi(f^*T_X^{log}(-x_{out}-x)) = A(\pmb\tau_1)\cdot c_1(T_X^{log}) - dim\text{ }X < -dim\text{ }X.\]

The equality above follows from Riemann-Roch and the inequality follows from assumption. Thus, it suffices to show that $dim\text{ }\mathfrak{M}^{ev(x_{out},x)}(\mathcal{X},\tau_1,z) \le dim\text{ }X$. Note that this dimension equals:

\[dim\text{ }\mathfrak{M}^{ev(x_{out},x)}(\mathcal{X},\tau_1,z) = dim (\mathfrak{M}^{ev(x_{out})}(\mathcal{X},\tau_1) \times_{\mathscr{P}(X,r)}^{fs} W)\times_{\underline{\mathcal{X}}} \underline{X} = dim (\mathfrak{M}^{ev}(\mathcal{X},\tau_1) \times_{\mathscr{P}(X,r)}^{fs} W) + dim\text{ }X.\]

On the other hand, since $r$ is contained in the $1$-skeleton of $B$, $dim\text{ } \pmb\sigma(v_{out}) \le 1$. Since the tropical evaluation map $ev_{v_{out}}^{gp}: \tau_1 \rightarrow \pmb\sigma(v_{out})$ must intersect the interior of $\pmb\sigma(v_{out})$, the evaluation must be surjective.  Thus, we have $dim\text{ }\tau_1 \ge codim\text{ }X_r$. By \cite{punc} Proposition $3.28$, we therefore have $dim\text{ } \mathfrak{M}^{ev}(\mathcal{X},\tau_1) \le dim\text{ }X - codim\text{ }X_r = dim\text{ }X_r$. 


To bound the dimension of $\mathfrak{M}^{ev}(\mathcal{X},\tau_1,z)$, we note that after removing all stratum $\mathfrak{M}_{\tau_1'}^{ev}\subset \mathfrak{M}^{ev}(\mathcal{X},\tau_1)$ with $\mathfrak{M}^{ev}_{\tau_1'} \times_{\mathscr{P}(X,r)} W = \emptyset$, we have $\mathfrak{M}^{ev}(\mathcal{X},\tau_1)$ is isomorphic to either $\mathfrak{M}(\mathcal{X},\tau_1) \times X^\circ$ or a $\mathbb{G}_m$-torsor over the product $\mathfrak{M}(\mathcal{X},\tau_1)\times X_r^\circ$, depending whether or not $r = 0$. Indeed, the first isomorphism is by definition. For the second isomorphism, we let $\mathcal{U}$ be the universal $\mathbb{G}_m$-torsor over $B\mathbb{G}_m$, we consider the morphism $(ev_{x_{out}}\circ \pi_1)^*(\mathcal{U})\otimes (p\circ\pi_2)^*(\mathcal{U}^{\vee}): \mathfrak{M}(\mathcal{X},\tau_1) \times X_r^\circ \rightarrow B\mathbb{G}_m$ with $p: X_r^\circ \rightarrow B\mathbb{G}_m$ the morphism classifying the log structure on $X_r$. The second isomorphisms follows by observing the following is a cartesian diagram, with the bottom horizontal morphism given above:

\[\begin{tikzcd}
\mathfrak{M}^{ev}(\mathcal{X},\tau_1) \arrow{r}\arrow{d}& Spec\text{ }\kk\arrow{d}\\
\mathfrak{M}(\mathcal{X},\tau_1) \times X_r^\circ \arrow[r]&B\mathbb{G}_m
\end{tikzcd}\]

Using this description, it follows that after removing the relevant strata from the source that $\mathfrak{M}^{ev}(\mathcal{X},\tau_1) \rightarrow X_r$ is flat, and hence $\mathfrak{M}^{ev}(\mathcal{X},\tau_1) \times_{X_r} z$ has dimension 
\[dim\text{ }\mathfrak{M}^{ev}(\mathcal{X},\tau_1) \times_{X_r} z = dim\text{ }\mathfrak{M}^{ev}(\mathcal{X},\tau_1) - dim\text{ }X_r.\] Additionally, it is straightforward to check that the previous fiber product of stacks is also the stack theoretic fiber product $\mathfrak{M}^{ev}(\mathcal{X},\tau_1)\times_{\underline{\mathscr{P}(X,r)}} \underline{W}$. Thus, we have the following dimension bound:
\[dim\text{ }\mathfrak{M}^{ev}(\mathcal{X},\tau_1) \times_{\underline{\mathscr{P}(X,r)}} \underline{W} \le dim\text{ }X_r  - dim\text{ }X_r=  0.\]
Since the dimension of fs fiber products is bounded above by the dimension of the fiber product of underlying stacks, the fs fiber product has dimension bounded above by $0$. From this inequality, we deduce $dim\text{ }\mathfrak{M}^{ev(x_{out},x)}(\mathcal{X},\tau_1,z) \le dim\text{ } X$, implying the desired negative virtual dimension.

\end{proof}

Thus, the spine of $\tau$ must map into $B \subset \Sigma(X)$. Note that if $[\mathscr{M}(X,\beta,z)]^{vir} \not= 0$, we must have $[\mathscr{M}(X,\pmb\tau,z)]^{vir} \not= 0$ for some realizable tropical type $\tau$ marked by $\beta$. As was just shown, this can only be the case if $D_i^*(s)= 0$ for all bad components $D_i$ of $D$. As a result, we have $c_1(T_X^{log})\cap \textbf{A}_2 = 0$. This restriction allows us to remove the nef/antinef assumption appearing in Theorem $37$ of \cite{RQC_no_log}, and the argument for this theorem in loc. cit. now shows $D_i^*(s) \ge 0$ for all components $D_i$ of $D$. Indeed, recall from Equation \ref{vdimorb} that the virtual dimension of the moduli space $\mathscr{M}^{orb}(X,\beta)$ is 

\[vdim\text{ }\mathscr{M}^{orb}(X,\beta) = c_1(T_X^{log})\cdot \textbf{A}_2 - 3 + n + dim\text{ }X -\sum_i e_i = dim\text{ }X - \sum_i e_i.\]
Since the contact order $r$ satisfies $D_i^*(r) < 0$ for only one divisor by the slope sensitivity condition, and the contact order $p_3$ satisfies $D_i^*(p_3) \ge 0$ for all components $D_i \subset D$, letting $e_s = |\{D_i\text{ }\mid{ }D_i^*(s) < 0\}|$, we have:

\[vdim\text{ }\mathscr{M}^{orb}(X,\beta) = dim\text{ }X - 1 - e_s.\]
If $\langle T_{s},[1]_{p_1},[pt]_{-r}\rangle_{\textbf{A}_2} \not= 0$, the virtual dimension of $\mathscr{M}^{orb}(X,\beta)$ must be at least $dim\text{ }X -1$, implying $e_s = 0$ and $T_s = [1]_s$.
\end{proof}

The above lemma allows us to conclude all the additional terms in the WDVV equation appearing on both sides of equation \ref{WDVV2} vanish in the non-minimal setting as well, and the resulting equality gives the desired associativity of the log mirror algebra. 
\end{proof}

\begin{remark}\label{rmk1}
While the examples given in Section $5$ show that the structure constants used to define $R_{(X,D)}$ and $R_{(X,D)}^{orb}$ are distinct, their difference is precisely because of the fact the latter algebra do not behave nicely under base-change, unlike the birational invariance property that $R_{(X,D)}$ established in \cite{pbirinv}. 

Additionally, it is noted in \cite{int_mirror} Remark $4.17$ that the no-tail lemma is crucial for associativity of the product, as it was in the above proof of associativity, and the failure of associativity in general should be encoded in a non-zero differential used in the definition of the analogue of symplectic cohomology in this setting.  
\end{remark}

\begin{example}

To better understand how the proof of associativity in this work relates with the proof found in \cite{int_mirror}, we consider the extended example $4.5$ of \cite{int_mirror}. This example features the blowup of $\mathbb{P}^2$ at an interior point on a coordinate axis, and $D$ is given by the strict transform of the toric boundary, we refer to the resulting log Calabi-Yau as $(X,D)$. The cone complex $\Sigma(X)$ can be identified with the cone complex associated with the fan of $\mathbb{P}^2$. Letting $D_1$ be the boundary component of the boundary which intersects the exceptional divisor, and $D_2$ and $D_3$ the remaining components, we consider the product of the theta functions $\vartheta_{p_1},\vartheta_{p_2},\vartheta_{p_3}$, with:

\[
p_1 = D_1+2D_2, \text{ } p_2 = D_2, \text{ } p_3= 2D_3, 
\]

We are particularly interested in the $\vartheta_{D_2}$ term of the resulting product of theta functions. This term is understood by intersection theory on the moduli space $\mathscr{M}(X,\beta,z)$, where $\beta$ fixes contact orders $p_1,p_2,p_3$ and $D_2^*$ and curve class $A = 2L - E$, L being the total transform of a line and $E$ the class of the exceptional curve.  This moduli space admits a stabilization map $\mathscr{M}(X,\beta,z) \rightarrow \overline{\mathscr{M}}_{0,4}$. The fiber over the boundary point of $\overline{\mathscr{M}}_{0,4}$ associated with the partition $(x_1,x_2|x_3,x_{out})$ is a non-reduced point of length $2$, while the fiber over the boundary point associated with the partition $(x_2,x_3|x_1,x_{out})$ is $\mathbb{P}^1$. The tropical types associated with the generic points of the fiber in the latter case is depicted in Figure \ref{m2}:

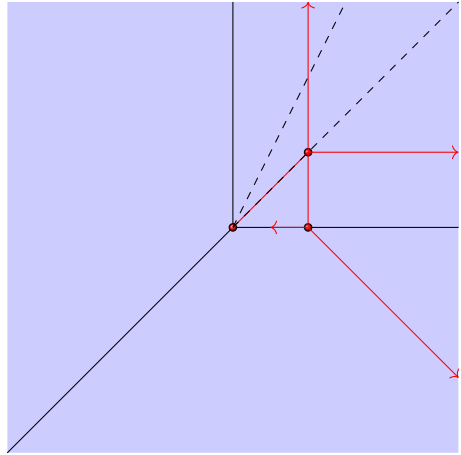
\begin{figure}[h]
\centering
\begin{tikzpicture}
\fill[white!80!blue, path fading = north ] (-6,0)--(-3,0)--(-3,3)--(-6,3)--cycle;
\fill[white!80!blue, path fading = south ] (-6,0)--(-3,0)--(-3,-3)--(-9,-3)--cycle;
\fill[white!80!blue, path fading = north] (-6,0)--(-9,-3)--(-9,3)--(-6,3)--cycle;
\draw[black] (-6,0)--(-6,3);
\draw[black] (-6,0)--(-3,0);
\draw[black] (-6,0)--(-9,-3);
\draw[ball color = red] (-6,0) circle (0.5mm);
\draw[ball color = red] (-5,1) circle (0.5mm);
\draw[ball color = red] (-5,0) circle (0.5mm);
\draw[-,color= red] (-6,0)--(-5,1);
\draw[->,color = red] (-5,1)--(-5,3);
\draw[->,color = red] (-5,1)--(-3,1);
\draw[-,color = red] (-5,1)--(-5,0);
\draw[->,color = red] (-5,0)--(-5.5,0);
\draw[->,color = red] (-5,0)--(-3,-2);
\draw[dashed] (-6,0)--(-3,3);
\draw[dashed] (-6,0)--(-4.5,3);
\end{tikzpicture}
\caption{The tropical curve associated with the generic point of $\mathscr{M}_2$}\label{m2}
\end{figure}

\begin{figure}[h]
\centering
\begin{tikzpicture}
\fill[white!80!blue, path fading = north ] (-6,0)--(-3,0)--(-3,3)--(-6,3)--cycle;
\fill[white!80!blue, path fading = south ] (-6,0)--(-3,0)--(-3,-3)--(-9,-3)--cycle;
\fill[white!80!blue, path fading = north] (-6,0)--(-9,-3)--(-9,3)--(-6,3)--cycle;
\draw[black] (-6,0)--(-6,3);
\draw[black] (-6,0)--(-3,0);
\draw[black] (-6,0)--(-9,-3);
\draw[ball color = red] (-6,0) circle (0.5mm);
\draw[ball color = red] (-5,1) circle (0.5mm);
\draw[ball color = red] (-4.5,2) circle (0.5mm);
\draw[ball color = red] (-5,0) circle (0.5mm);
\draw[-,color= red] (-6,0)--(-5,1);
\draw[->,color = red] (-4.5,2)--(-4.5,3);
\draw[->,color = red] (-4.5,2)--(-3,2);
\draw[-,color = red] (-5,1)--(-4.5,2);
\draw[-,color = red] (-5,1)--(-5,0);
\draw[->,color = red] (-5,0)--(-5.5,0);
\draw[->,color = red] (-5,0)--(-3,-2);
\draw[dashed] (-6,0)--(-3,3);
\draw[dashed](-6,0)--(-4.5,3);
\draw[dashed](-6,0)--(-3,-3);

\fill[white!80!blue, path fading = north ] (1,0)--(4,0)--(4,3)--(1,3)--cycle;
\fill[white!80!blue, path fading = south ] (1,0)--(4,0)--(4,-3)--(-2,-3)--cycle;
\fill[white!80!blue, path fading = north] (1,0)--(-2,-3)--(-2,3)--(1,3)--cycle;
\draw[black] (1,0)--(1,3);
\draw[black] (1,0)--(4,0);
\draw[black] (1,0)--(-2,-3);
\draw[ball color = red] (1,0) circle (0.5mm);
\draw[ball color = red] (1.9,0) circle (0.5mm);
\draw[ball color = red] (2,1) circle (0.5mm);
\draw[ball color = red] (2,2.5) circle (0.5mm);
\draw[-,color= red] (1,0)--(2,1);
\draw[-,color = red] (2,2.5)--(1.9,2.5);
\draw[-,color = red] (1.9,2.5)--(1.9,0);
\draw[-,color = red] (2,1)--(2,2);
\draw[->,color = red] (1.9,0)--(1.4,0);
\draw[->,color = red] (1.9,0)--(3.9,-2);
\draw[->,color = red] (2,2)--(2,3);
\draw[->,color = red] (2,1)--(4,1);
\draw[dashed] (1,0)--(4,3);
\draw[dashed](1,0)--(2.5,3);
\draw[dashed](1,0)--(4,-3);

\end{tikzpicture}
\caption{Tropical curves associated to lower dimensional strata of $\mathscr{M}_2$ which contribute to $\vartheta_{p_1}\vartheta_{p_2}\vartheta_{p_3}[\vartheta_{D_2}z^A]$.}\label{m2s}
\end{figure}

This discrepancy is fixed in \cite{int_mirror} by imposing a logarithmic condition on the image under the stabilization map which requires the length of the unique edge in the tropicalization to be large. This results in an integral on $\mathscr{M}_2$ which is localized to the two strata associated with the tropical data depicted in Figure \ref{m2s}.

These two strata are also identified from the perspective taken in this paper, as the natural cutting edge in Figure \ref{m2} does not have a non-negative contact order in the blowup required to ensure slope sensitivity, while the tropical types depicted in Figure \ref{m2s} do have cutting edges with non-negative contact orders. 
\end{example}

In a similar fashion, we also deduce the weak Frobenius structure theorem for all targets considered in this section. The basic structure of the argument is similar to that for the associativity of the mirror algebra, except instead of considering the WDVV relation, we instead consider the following topological recursion relation.

\begin{equation}\label{TRR}
N^\beta_{p_1,\ldots,p_m,0} = \sum_{\textbf{A}_1 + \textbf{A}_2 = \textbf{A} \in H_2(X),\text{ }s \in \ZZ^n,\text{ }k,S_1,S_2} \langle[pt]_0\psi^{m-3},\prod_{j \in S_1}[1]_{p_j},T_{-s,k}\rangle_{\textbf{A}_1}\langle T^k_{s},[1]_{p_1},[1]_{p_2},\prod_{j \in S_2} [1]_{p_j}\rangle_{\textbf{A}_2}.
\end{equation}

In the above equation, the pair $(S_1,S_2)$ vary over all pairs of disjoint subsets of $\{1,\ldots,m\}$ such that $S_1\cup S_2 = \{3,\ldots,m\}$. As before, we may assume that all of the corresponding moduli spaces of punctured log maps are all slope sensitive. We wish to show the following:

\begin{lemma}
$ \langle[pt]_0\psi^{m-3},\prod_{j \in S_1}[1]_{p_j},T_{-s,k}\rangle_{\textbf{A}_1} = 0$ unless $T_{-s,k} = [1]_{-s}$, $-s$ is a positive contact order, and $|S_1| = m-2$.
\end{lemma}

\begin{proof}
By Corollary \ref{logorb}, we may show the given orbifold invariant involving a point constraint vanishes by considering an integral on the moduli space of point constrained log stable maps $\mathscr{M}(X,\beta,z)$, with $z \in X\setminus D$ generic, and $\beta$ a decorated tropical type specifying contact order $0,-s$, and $p_j$ for $j \in S_1$. As before, we consider the virtual pullback along the map $\mathscr{M}(X,\beta,z) \rightarrow \mathfrak{M}^{ev}(\mathcal{X},\beta,z)$ of strata of $\mathfrak{M}^{ev}(\mathcal{X},\beta,z)$, and $\eta \in \mathfrak{M}^{ev}(\mathcal{X},\beta,z)$ the generic point of the statum. Letting $\tau$ be the tropical type of the log map $C_\eta \rightarrow \mathcal{X}$ associated with $\eta$, note that for the vertex $v_{out} \in V(G_\tau)$ containing the leg with contact order $0$, the point constraint ensures that $\tau$ is contributing only if $\pmb\sigma(v_{out})$ is the zero cone in $\Sigma(X)$. We additionally define an modification of the obstruction theory associated with the morphism $\mathscr{M}(X,\beta,z) \rightarrow \mathfrak{M}^{ev}(\mathcal{X},\beta,z) $ to be the dual of
\[ R\pi_*f^*T_X^{log}(-x_{out})\oplus L_{x_{out}}^{m-3}[1],\]
 with $L_{x_{out}}$ the cotangent line to $x_{out}$. The new virtual fundamental class $[\mathscr{M}(X,\beta,z)]^{vir'}$ is related to the original virtual fundamental class by 
 \[[\mathscr{M}(X,\beta,z)]^{vir'} = [\mathscr{M}(X,\beta,z)]^{vir}\cap \psi_{x_{out}}^{m-3}.\]
 In particular, the invariant of interest is the degree of the Chow class $[\mathscr{M}(X,\beta,z)]^{vir'}$. 

%
%

We now prove the following lemma, which is a slight modification of Lemma \ref{spine} above:

\begin{lemma}
Let $\tau$ be a tropical type with legs $L,l_{out},\ldots,l_m$, $m > 1$, having contact orders $D_i^*(u_{l_i}) = 0$ for all bad divisors $D_i$, and a vertex $v \in V(G_\tau)$ with $\pmb\sigma(v) \in \mathscr{P}(X)$, contained in the legs $l_{out},\ldots,l_m$. If $[\mathscr{M}(X,\pmb\tau,z,C)]^{vir} \not= 0$, then for all edges and legs $e \in E(G_\tau) \cup L(G_\tau)$ contained in the convex hull of $v$ and $L$, we have $D_i^*(u_e) = 0$ for any bad divisor $D_i$. 
\end{lemma}

\begin{proof}
As in Lemma \ref{spine}, we prove the proposition by induction on the length of a chain of edges in the spine, with the labeling of edges and the base case as in Lemma \ref{spine}. Assuming the inductive hypothesis for the chain of length $k-1$, we cut $\tau$ along an edge $e_k \in E(G_\tau)$ in the spine, producing two tropical type $\tau_1$ and $\tau_2$ with $\tau_1$ containing the vertex $v$. By the same argument at the start of Lemma \ref{spine}, we have $A(\tau_1)\cdot c_1(T_X^{log}) \le 0$. Moreover, we observe by Corollary $7.3$ of \cite{punc} that we have the following cartesian diagram:

\[\begin{tikzcd}
\mathscr{M}(X,\tau,z) \arrow{r}\arrow{d}& \mathscr{M}(X,\tau_1,z)\times\mathscr{M}(X,\tau_2) \arrow{d}\\
\mathfrak{M}^{ev(x_{out})}(\mathcal{X},\tau,z) \arrow{r} & \mathfrak{M}^{ev(x_{out},x)}(\mathcal{X},\tau_1,z)\times \mathfrak{M}^{ev(x)}(\mathcal{X},\tau_2).
\end{tikzcd}\]

Unlike in Lemma \ref{spine}, the obstruction theories on the left vertical arrow is the one defined preceeding this lemma, and the obstruction theory for the right vertical arrow is the product obstruction theory, with the usual one for the moduli stacks associated with $\tau_2$, and the dual of $R\pi_*f^*T_X^{log}(-x_{out}-x) \oplus L_{x_{out}}^{m-3}[1]$ for $\tau_1$. Since the cotangent line for $x_{out}$ on $\mathscr{M}(X,\tau_1,z)$ clearly pulls back to the cotangent line for $x_{out}$ on $\mathscr{M}(X,\tau,z)$, the obstruction theories are compatible. Furthermore, we have the relative virtual dimension of $\mathscr{M}(X,\pmb\tau_1,z) \rightarrow \mathfrak{M}^{ev(x_{out},x)}(\mathcal{X},\tau_1,z)$ is less than or equal to $-dim\text{ }X-m+3$, with equality only when the the conclusion of the lemma holds. Hence, we must show as before that $\mathfrak{M}^{ev(x_{out},x)}(\mathcal{X},\tau_1,z) \le dim\text{ }X +m -3$. This follows from the analogous argument given in Lemma \ref{spine}, with the only difference being Proposition $3.28$ of \cite{punc} gives $dim\text{ }\mathfrak{M}^{ev}(\mathcal{X},\tau_1)\le dim\text{ }X + m - 3$. We conclude that $vdim\text{ }\mathscr{M}(X,\tau_1,z)< 0$ when there is a bad divisor $D_i$ such that $D_i^*(u_{e_k}) \not= 0$, and hence by \cite{int_mirror} Theorem $A.13$, $[\mathscr{M}(X,\pmb\tau,z)]^{vir} = 0$. 
\end{proof}

By the above lemma, we must have $D_i^*(-s) = 0$ for all bad divisors $D_i$. In particular, since $c_1(T_X^{log}) = \sum_i -a_iD_i$ with $a_i\ge 0$ and $a_i > 0$ for only $i$ such that $D_i$ is a bad divisor, by Corollary $1.14$ of \cite{int_mirror}, we have  $c_1(T_X^{log})\cdot \textbf{A}_1 = 0$. Now recall the virtual dimension of $\mathscr{M}^{orb}(X,\beta)$ is:
\[vdim\text{ }\mathscr{M}^{orb}(X,\beta) = c_1(T_X^{log})\cdot \textbf{A}_1 - 3 + n + dim\text{ }X - \sum_i e_i = -3 + |S_1| + 2 +dim\text{ }X - \sum_{i=1}^k e_i.\]
For the invariant $\langle[pt]_0\psi^{m-3},\prod_{j \in S_1}[1]_{p_j},T_{-s,k}\rangle_{\textbf{A}_1}$ to be non-zero, we must have 
\[vdim\text{ }\mathscr{M}(X,\beta) \ge m-3+dim\text{ }X.\]
Since $|S_1| \le m-2$, we must have $e_i = 0$ for all $i$ and $|S_1| = m-2$. With these constraints, the only possible insertion we can have is $T_{-s,k} = [1]_{-s}$ as required. 
\end{proof}

Using the above lemma, Equation \ref{TRR} reduces to:

\[N^\textbf{A}_{p_1,\ldots,p_m,0} = \sum_{(\textbf{A}_1,\textbf{A}_2,s)} N^{\textbf{A}_1}_{p_3,\ldots,p_m,s,0}N_{p_1,p_2,s}^{\textbf{A}_2}\]
A straightforward induction on size of $m$ now shows:
\[N^\textbf{A}_{p_1,\ldots,p_m,0} = \sum_{(\textbf{A}_1,\ldots,\textbf{A}_{m-1},s_1,\ldots,s_{m-2})} N_{p_1,p_2,s_1}^{\textbf{A}_1}(\prod_{i = 2}^{m-2} N^{\textbf{A}_i}_{s_{i-1},p_{i+1},s_i})N_{s_{m-2},p_k,0}^{\textbf{A}_{m-1}}.\]
The right-hand side of the above relation is the expression for the $z^\textbf{A}\vartheta_0$ term of $\vartheta_{p_1}\cdots\vartheta_{p_m}$ found by repeatedly applying the product rule for theta functions, giving the desired equality of Theorem \ref{mthm2}.

\begin{remark}
The utilities of both relative theories can be observed in this paper. On the orbifold side, there are straight out of the box relations which sidestep many otherwise delicate gluing arguments on the log GW side of the correspondence. On the other hand, the orbifold theory is not log \'etale invariant unlike the log theory, leading to the counterexample to log \'etale invariance presented in this paper. On the logarithmic side, the relationship between the virtual geometry of the moduli spaces of log stable maps and the corresponding tropical moduli problems allow us to refine the invariants of interest by specifying tropical behavior, which is not possible using only orbifold Gromov-Witten invariants. As a result, we believe the correct takeaway is that both theories yield interesting and distinct insights into the Gromov-Witten theory of pairs $(X,D)$, and there is much potential in combining techniques from these two areas.

\end{remark}

%
%

%
%

\nocite{*}
\bibliographystyle{amsalpha}
\bibliography{assoc}

\end{document}